\documentclass[14pt,reqno,a4paper]{amsart}
\usepackage{hyperref}
\usepackage[left=0.7in, top=1in, bottom=1.7in, right=0.7in]{geometry}
\usepackage{enumerate}
\usepackage{amsmath,amssymb,amsthm,mathrsfs,amsfonts,dsfont}
\usepackage{graphicx}
\usepackage{tikz}
\usetikzlibrary{graphs,graphs.standard}
\usetikzlibrary{shapes}
\usetikzlibrary{plotmarks}
\usetikzlibrary{arrows}
\usetikzlibrary{positioning}
\theoremstyle{definition}
\newtheorem{defn}{Definition}[section]

\newtheorem{prop}[defn]{Proposition}

\newtheorem{thm}[defn]{Theorem}

\newtheorem{lem}[defn]{Lemma}

\newtheorem{question}[defn]{Question}
\newtheorem{remark}[defn]{Remark}
\newtheorem{claim}[defn]{claim}

\title[Partition models, Permutations of infinite sets, and weak forms of $\mathsf{AC}$]{Partition models, Permutations of infinite sets without fixed points, and weak forms of $\mathsf{AC}$}

\author{A\MakeLowercase{mitayu} B\MakeLowercase{anerjee}}

\address{Alfr\'ed R\'enyi Institute of Mathematics, Reáltanoda utca 13-15, Budapest-1053, Hungary}
\email{banerjee.amitayu@gmail.com}
\makeatletter
\@namedef{subjclassname@2020}{\textup{2020} Mathematics Subject Classification}
\makeatother
\begin{document}

\maketitle
\begin{abstract}
We study new relations of the following statements with weak choice principles in $\mathsf{ZF}$ (Zermelo--Fraenkel set theory without the Axiom of Choice ($\mathsf{AC}$)) and $\mathsf{ZFA}$ ($\mathsf{ZF}$ with the axiom of extensionality weakened to allow the existence of atoms).
\begin{itemize}
\item There does not exist an infinite Hausdorff space $X$ such that every infinite subset of $X$ contains an infinite compact subset.
\item For every set $X$ there is a set $A$ such that there exists a choice function on the collection $[A]^2$ of two-element subsets of $A$ and satisfying $\vert X\vert\leq \vert2^{A^{2}}\vert$.

\item If a field has an algebraic closure then it is unique up to isomorphism.  

\item For every infinite set $X$, there exists a permutation of $X$ without fixed points.

\item $\mathsf{vDCP}$ (Van Douwen’s  Choice  Principle).

\item Any infinite locally finite connected graph has a spanning subgraph omitting $K_{2,n}$ for any $2< n\in\omega$.

\item Any infinite locally finite connected graph has a spanning $m$-bush for any even integer $m\geq 4$.
\end{itemize}
We also study the new status of different weak choice principles in the finite partition model (a type of permutation model of $\mathsf{ZFA+\neg AC}$) introduced by Benjamin Baker Bruce in 2016.
\end{abstract}

\section{Introduction and Abbreviations}

%%%%%%%%%%%
\subsection{\textbf{Forms 269, 233, and 304}}
We study the status of certain weak choice principles in a recent permutation
model constructed by Halbeisen and Tachtsis in \cite[Theorem 8]{HT2020}.
  
\subsubsection{Necessary weak choice forms}
Let $X$ and $Y$ be sets. We write $\vert X\vert \leq \vert Y\vert$ if there is an injection $f: X \rightarrow Y$.
\begin{itemize}

\item \cite[Form 269]{HR1998}: 
For every set $X$ there is a set $A$ such that there exists a choice function on the collection $[A]^2$ of two-element subsets of $A$ and satisfying $\vert X\vert\leq \vert2^{A^{2}}\vert$. 

\item \cite[Form 233]{HR1998}: If a field has an algebraic closure then it is unique up to isomorphism. 

\item \cite[Form 304]{HR1998}: There does not exist an infinite Hausdorff space $X$ such that every infinite subset of $X$ contains an infinite compact subset.

\item $\mathsf{AC^{LO}}$ \cite[Form 202]{HR1998}: Every linearly ordered family of non-empty sets has a 
choice function.

\item $\mathsf{LW}$ \cite[Form 90]{HR1998}: Every linearly ordered set can be well-ordered.    
    
\item $\mathsf{AC^{WO}}$ \cite[Form 40]{HR1998}: Every well-orderable set of non-empty sets has a choice 
function.  

\item $\mathsf{AC_{n}^{-}}$ for each $n \in \omega, n \geq 2$ \cite[Form 342($n$)]{HR1998}: Every infinite family $\mathcal{A}$ of $n$-element sets has a partial choice function, i.e., $\mathcal{A}$ has an infinite subfamily $\mathcal{B}$ with a choice function. 

\item The {\em Chain/Antichain Principle}, $\mathsf{CAC}$ \cite[Form 217]{HR1998}: Every infinite partially ordered set (poset) has an infinite chain or an infinite antichain.

\item \cite[Form 64]{HR1998}: There are no amorphous sets (An inﬁnite set $X$ is {\em amorphous} if $X$ cannot be written as a disjoint union of two infinite subsets).
\end{itemize}

\subsubsection{Results} Pincus proved that \cite[Form 233]{HR1998} holds in the basic Fraenkel model (cf. \cite[Note 41]{HR1998}). It is also known that in the basic Fraenkel model
\cite[Form 269]{HR1998} fails, whereas \cite[Form 304]{HR1998} holds (cf. \cite[Notes 91, 116]{HR1998}). Fix a natural number $2\leq n\in\omega$. Halbeisen--Tachtsis \cite[Theorem 8]{HT2020} constructed a permutation model (we denote by $\mathcal{N}_{HT}^{1}(n)$) where $\mathsf{AC_{n}^{-}}$ fails but $\mathsf{CAC}$ holds. We prove the following (cf. Theorem 3.1, Theorem 3.3):

\begin{enumerate}
    \item  In $\mathsf{ZFA}$, $\mathsf{AC^{LO}}$ does not imply Form 269. 
    
    \item Form 269 fails in $\mathcal{N}_{HT}^{1}(n)$ whereas Form 233 and Form 304 hold in $\mathcal{N}_{HT}^{1}(n)$. Consequently, for any integer $n\geq 2$, Form 233 (similarly Form 304) neither implies $\mathsf{AC_{n}^{-}}$
    nor implies `There are no amorphous sets'
    in $\mathsf{ZFA}$.

\end{enumerate}
%%%%%%%%%%%%%%%%%%
\subsection{Partition models and permutations of infinite sets} 
We study the failure of certain weak choice principles in the finite partition model introduced by Bruce in \cite{Bru2016}. 

\subsubsection{Necessary weak choice forms and abbreviations} 
A set $X$ is {\em almost even} if there is a permutation $f$ of $X$ without fixed points and such that $f^{2} = $id$_{X}$.
\begin{itemize}
    \item $\mathsf{AC_{n}}$ for each $n \in \omega, n \geq 2$ \cite[Form 61]{HR1998}: Every family of $n$-element sets has a choice function.
    \item $\mathsf{DF = F}$ \cite[Form 9]{HR1998}: Every Dedekind-finite set is finite (A set $X$ is called {\em Dedekind-finite} if $\aleph_{0} \not
    \leq \vert X\vert$ i.e., if there is no one-to-one function $f : \omega \rightarrow X$. Otherwise, $X$ is called {\em Dedekind-infinite}).
    \item $\mathsf{W_{\aleph_{\alpha}}}$ (cf. \cite[Chapter 8]{Jec1973}): For every $X$, either $\vert X\vert \leq \aleph_{\alpha}$ or $\vert X\vert \geq \aleph_{\alpha}$. We recall that $\mathsf{W_{\aleph_{0}}}$ is equivalent to $\mathsf{DF=F}$ in $\mathsf{ZF}$.
    
    \item $\mathsf{DC_{\kappa}}$ 
    where $\alpha$ is the ordinal such that $\kappa=\aleph_{\alpha}$
    \cite[Form 87($\alpha$)]{HR1998}: Let $S$ be a non-empty set and let $R$ be a binary relation such that for every $\beta<\kappa$ and every $\beta$-sequence $s =(s_{\epsilon})_{\epsilon<\beta}$ of elements of $S$ there exists $y \in S$ such that $s R y$. Then there is a function $f : \kappa \rightarrow S$ such that for every $\beta < \kappa$, $(f\restriction \beta) R f(\beta)$. We note that $\mathsf{DC_{\aleph_{0}}}$ is a reformulation of $\mathsf{DC}$ (the principle of Dependent Choices \cite[Form 43]{HR1998}). We denote by $\mathsf{DC_{<\lambda}}$ the assertion $(\forall\eta < \lambda)\mathsf{DC_{\eta}}$.

    \item  $\mathsf{UT(WO, WO, WO)}$ \cite[Form 231]{HR1998}: The union of a well-ordered collection of well-orderable sets is well-orderable.
    
    \item \textbf{$\mathsf{(\forall\alpha)UT(\aleph_{\alpha},\aleph_{\alpha},\aleph_{\alpha})}$} \cite[Form 23]{HR1998}: For every ordinal $\alpha$, if $A$ and every member of $A$ has cardinality $\aleph_{\alpha}$, then $\vert \bigcup A\vert=\aleph_{\alpha}$.
    
    \item The {\em Axiom of Multiple Choice, $\mathsf{MC}$} \cite[Form 67]{HR1998}: Every family $\mathcal{A}$ of non-empty sets has a multiple choice function, i.e., there is a function $f$ with domain $\mathcal{A}$ such that for every $A \in \mathcal{A}$, $f(A)$ is a non-empty finite subset of $A$.
    
    \item $\mathsf{\leq\aleph_{0}}$-$\mathsf{MC}$ (cf. \cite[section 1]{HST2016}): For any family $\{A_{i}:i\in I\}$ of non-empty sets, there is a function $F$ with domain $I$ such that for all $i\in I$, $F(i)$ is a non-empty countable (i.e., finite or countably infinite) 
    subset of $A_{i}$.
    \item \cite[Form 3]{HR1998}: For every infinite set $X$, the sets $X$ and $X \times \{0, 1\}$ are equipotent  (i.e., there exists a bijection $f : X \rightarrow X\times\{0, 1\}$).

    \item $\mathsf{ISAE}$ (cf. \cite[section 2]{Tac2019}): Every infinite set is almost even (i.e., for every infinite set $X$, there is a permutation $f$ of $X$ without fixed points and such that $f^{2} = $id$_{X}$).
    \item $\mathsf{EPWFP}$  (cf. \cite[section 2]{Tac2019}): For every infinite set $X$, there exists a permutation of $X$ without fixed points.

\item For a set $A$, Sym$(A)$ and  FSym$(A)$ denote the set of all permutations of $A$ and the set of all $\phi \in$ Sym$(A)$ such that $\{x \in A : \phi(x) \neq x\}$ is finite, respectively (cf. \cite[section 2]{Tac2019}).

\item For a set $A$ of size at least $\aleph_{\alpha}$, $\aleph_{\alpha}$Sym$(A)$ denote the set of all $\phi \in$ Sym$(A)$ such that $\{x \in A : \phi(x) \neq x\}$ has cardinality at most $\aleph_{\alpha}$.

\item $\mathsf{MA(\kappa)}$ for a well-ordered cardinal $\kappa$ (cf. \cite[section 1]{Tac2016}): If $(P,<)$ is a nonempty, c.c.c. quasi order and if $\mathcal{D}$ is a family of $\leq\kappa$ dense sets in $P$, then there is a filter $\mathcal{F}$ of $P$ such that $\mathcal{F}\cap D\not=\emptyset$ for all $D\in \mathcal{D}$. 
\end{itemize}

\subsubsection{Results} Bruce \cite{Bru2016} constructed the finite partition model $\mathcal{V}_{p}$, which is a variant of the basic Fraenkel model (labeled as Model $\mathcal{N}_{1}$ in \cite{HR1998}). Many, but not all, properties of $\mathcal{N}_{1}$ transfer to $\mathcal{V}_{p}$. In particular, Bruce proved that the set of atoms has no amorphous subset in $\mathcal{V}_{p}$ unlike in $\mathcal{N}_{1}$, whereas $\mathsf{UT(WO, WO, WO)}$, $\mathsf{\neg AC_{2}}$, and $\mathsf{\neg (DF=F)}$ hold in $\mathcal{V}_{p}$ as in $\mathcal{N}_{1}$. At the end of the paper, Bruce asked which other choice principles hold in $\mathcal{V}_{p}$ (cf. \cite[section 5]{Bru2016}). We study the status of some weak choice principles in $\mathcal{V}_{p}$. We also study the status of some weak choice principles in a variant of $\mathcal{V}_{p}$ mentioned in \cite[section 5]{Bru2016}. In particular, let $A$ be an uncountable set of atoms, let $\mathcal{G}$ be the group of all permutations of $A$, and let the supports be countable partitions of $A$. We call the corresponding permutation model $\mathcal{V}^{+}_{p}$. At the end of the paper, Bruce asked about the status of different weak choice forms in $\mathcal{V}^{+}_{p}$. Fix any integer $n\geq 2$. We prove the following (cf. Theorem 4.5, Proposition 4.7, Theorem 4.8):
\begin{enumerate}
    \item $\mathsf{W_{\aleph_{\alpha+1}}}$ implies `for any set $X$ of size $\aleph_{\alpha+1}$, Sym(X) $\neq$ $\aleph_{\alpha}$Sym(X)' in $\mathsf{ZF}$.

   \item If $X\in \{\mathsf{EPWFP}, \mathsf{MA(\aleph_{0})}, \mathsf{AC_{n}},  \mathsf{MC}, \mathsf{\leq\aleph_{0}}$-$\mathsf{MC}\}$, then $X$ fails in $\mathcal{V}_{p}$.
   
    \item If $X\in \{\mathsf{EPWFP}, \mathsf{AC_{n}}, \mathsf{W_{\aleph_{1}}}\}$, then $X$ fails in $\mathcal{V}^{+}_{p}$.
\end{enumerate}
%%%%%%%%%%%%%%%%%
\subsection{Van Douwen’s Choice Principle in permutation models} 

Howard, Saveliev, and Tachtsis \cite[p.175]{HST2016} gave an argument to prove that Van Douwen’s Choice Principle ($\mathsf{vDCP}$) holds in the basic Fraenkel model. We modify the argument slightly to prove that $\mathsf{vDCP}$ holds in two recently constructed permutation models (cf. section 5).

\subsubsection{Necessary weak choice forms and abbreviations} 
\begin{itemize}
    \item $\mathsf{UT(\aleph_{0}, \aleph_{0}, cuf)}$ \cite[Form 420]{HR}: Every countable union of countable sets is a cuf set (A set $X$ is called a {\em cuf set} if $X$ is expressible as a countable union of finite sets). 
    
    \item $\mathsf{M(IC, DI)}$ (cf. \cite{KTW2021}): Every infinite compact metrizable space is Dedekind-infinite.
    
    \item $\mathsf{MC(\aleph_{0}, \aleph_{0})}$ \cite[Form 350]{HR1998}: Every denumerable family of denumerable sets has a multiple choice function.
    \item {\em Van Douwen’s Choice Principle, $\mathsf{vDCP}$}: Every family $X = \{(X_{i}, \leq_{i}) : i \in I\}$ of linearly ordered sets
isomorphic with $(\mathbb{Z}, \leq)$ ($\leq$ is the usual ordering on $\mathbb{Z}$) has a choice function.
\end{itemize}

\subsubsection{Results} Howard and Tachtsis \cite[Theorem 3.4]{HT2021} proved that the statement $\mathsf{LW} \land \mathsf{\neg MC(\aleph_{0}, \aleph_{0})}$ has a permutation model, say $\mathcal{M}$.
The authors of \cite[proof of Theorem 3.3]{CHHKR2008} constructed a permutation model $\mathcal{N}$ where $\mathsf{UT(\aleph_{0}, \aleph_{0}, cuf)}$ holds. Keremedis, Tachtsis, and Wajch \cite[Theorem 13]{KTW2021} proved that $\mathsf{LW}$ holds and $\mathsf{M(IC, DI)}$ fails in $\mathcal{N}$. We prove the following (cf. Proposition 5.1):
\begin{enumerate}
    \item $\mathsf{vDCP}$ holds in $\mathcal{N}$ and $\mathcal{M}$.
\end{enumerate}

\subsection{Spanning subgraphs} Fix any $2< n\in \omega$ and any even integer $4\leq m\in \omega$.  Delhomm\'{e}--Morillon \cite[Corollary 1, Remark 1]{DM2006} proved that $\mathsf{AC}$ is equivalent to {\em `Every bipartite connected graph has a spanning subgraph omitting $K_{n,n}$'} as well as {\em `Every connected graph admits a spanning $m$-bush'}. We study new relations between variants of the above statements and weak forms of $\mathsf{AC}$. 

\subsubsection{Necessary weak choice forms and abbreviations} 

\begin{itemize}
    
\item Let $n \in \omega \backslash \{0, 1\}$. 
$\mathsf{AC^{\omega}_{\leq n}}$: Every denumerable family of non-empty sets, each with at most $n$
elements, has a choice function.

\item  $\mathsf{AC_{fin}^{\omega}}$ \cite[Form 10]{HR1998}:  Every denumerable family of non-empty finite sets has a choice function.
\item $\mathsf{AC_{WO}^{WO}}$ \cite[Form 165]{HR1998}:  Every well-orderable family of non-empty well-orderable sets has a choice function.

\item Let $n \in \omega \backslash \{0, 1\}$. 
$\mathsf{AC^{WO}_{\leq n}}$: Every well-orderable family of non-empty sets, each with at most $n$
elements, has a choice function.
\end{itemize}

Fix any $2< k,n\in \omega$ and any even integer $4\leq m\in \omega$. We introduce the following abbreviations.
\begin{itemize}
\item $\mathcal{Q}_{lf,c}^{n}$: Any infinite locally finite connected graph has a spanning subgraph omitting $K_{2,n}$.
\item $\mathcal{Q}_{lw,c}^{k,n}$: Any infinite locally well-orderable connected graph has a spanning subgraph omitting $K_{k,n}$.
\item $\mathcal{P}_{lf,c}^{m}$: Any infinite locally finite connected graph has a spanning $m$-bush. 
\end{itemize}

Let $G$ be a graph. We denote by $P_{G}$, the class of those infinite graphs whose only components are isomorphic with  $G$. For any graph $G_{1}=(V_{G_{1}}, E_{G_{1}})\in P_{G}$, we construct a graph $G_{2}=(V_{G_{2}}, E_{G_{2}})$ as follows: Pick a $t\not\in V_{G_{1}}$. Let $V_{G_{2}}=\{t\}\bigcup V_{G_{1}}$, $E_{G_{1}}\subseteq E_{G_{2}}$ and for each $x\in V_{G_{1}}$, let $\{t,x\}\in E_{G_{2}}$. We denote by $P'_{G}$, the class of graphs of the form $G_{2}$.

\subsubsection{Results} Fix $2< n,k\in \omega$ and $2\leq p,q< \omega$. We prove the following in $\mathsf{ZF}$ (cf. Proposition 6.5 and Proposition 6.6): 

\begin{enumerate}
\item $\mathsf{AC_{\leq n-1}^{\omega}}$ + $\mathcal{Q}_{lf,c}^{n}$ is equivalent to $\mathsf{AC_{fin}^{\omega}}$.

\item $\mathsf{UT(WO,WO,WO)}$ implies $\mathsf{AC_{\leq n-1}^{WO}}$ + $\mathcal{Q}_{lw,c}^{k,n}$ and the later implies $\mathsf{AC_{WO}^{WO}}$.

\item $\mathcal{P}_{lf,c}^{m}$ is equivalent to $\mathsf{AC_{fin}^{\omega}}$ for any even integer $m\geq 4$.

\item $\mathsf{AC_{k^{k-2}}}$ implies {\em `Every graph from the class $P'_{K_{k}}$ has a spanning tree}'.

\item $\mathsf{AC_{k}}$ implies {\em `Every graph from the class $P'_{C_{k}}$ has a spanning tree}'.

\item $(\mathsf{AC_{p^{q-1}q^{p-1}}+AC_{p+q}})$ implies {\em `Every graph from the class $P'_{K_{p,q}}$ has a spanning tree}'.
\end{enumerate}
%%%%%%%%%%%%%%
\section{Basics}
\begin{defn} {\textbf{(Topological definitions)}}
Let $\textbf{X}=(X,\tau)$ be a topological space. We say $\textbf{X}$ is {\em Baire} if for every countable
family $\mathcal{O} = \{O_{n} : n \in \omega\}$ of dense open subsets of $X$,
$\bigcap \mathcal{O}$ is nonempty and dense. We say $\textbf{X}$ is {\em compact} if for every $U \subseteq \tau$ such that $\bigcup U = X$ there is a finite subset $V\subseteq U$ such that $\bigcup V = X$.
The space $\textbf{X}$ is called a {\em Hausdorff (or $T_{2}$-)space} if any two distinct points in $X$ can be separated by disjoint open sets, i.e. if $x$ and $y$ are distinct points of $X$, then there exist disjoint open sets $U_{x}$ and $U_{y}$ such that $x\in U_{x}$ and $y\in U_{y}$. 
\end{defn}

\begin{defn}{\textbf{(Algebraic definitions)}} 
A {\em permutation} on a set $Y$ is a one-to-one correspondence from $Y$ to itself. The set of all permutations on $Y$, with operation defined to be the composition of mappings, is the {\em symmetric group} of $Y$, denoted by $Sym(Y)$. Let $X$ be a finite set. Fix $r\leq \vert X\vert$. A permutation $\sigma \in Sym(X)$ is a {\em cycle of length $r$} if
there are distinct elements $i_{1},...,i_{r}\in X$ such that $\sigma(i_{1})=i_{2},\sigma(i_{2})=i_{3},...,\sigma(i_{r})=i_{1}$ and $\sigma(i)=i$ for all $i\in X\backslash \{i_{1},..., i_{r}\}$. In this case we write $\sigma=(i_{1},...,i_{r})$. A cycle of length 2 is called a {\em transposition}. We recall that $(i_{1},...,i_{r})=(i_{1}, i_{r})(i_{1},i_{r-1})...(i_{1}, i_{2})$. So, every permutation can be written as a product of transpositions. A permutation $\sigma\in Sym(X)$ is an {\em even permutation} if it can be written as the product of an even number of transpositions; otherwise,
it is an {\em odd permutation}. The {\em alternating group} of $X$, denoted by $Alt(X)$, is the group of all even permutations in $Sym(X)$. If $\mathcal{G}$ is a group and $X$ is a set, an {\em action of $\mathcal{G}$ on $X$} is a group homomorphism $F: \mathcal{G} \rightarrow Sym(X)$. If a group $\mathcal{G}$ acts on a set $X$, we say $\text{Orb}_{\mathcal{G}}(x)=\{gx:g\in \mathcal{G}\}$ is the orbit of $x \in X$ under the action of $\mathcal{G}$. We recall that different orbits of the action are disjoint and form a partition of $X$ i.e., $X=\bigcup \{\text{Orb}_{\mathcal{G}}(x): x\in X\}$.
Let $\{G_{i}:i\in I\}$ be an indexed collection of groups. 
Define the following set.
\begin{equation}\prod_{i\in I}^{weak}G_{i}=\left\{f:I\rightarrow \bigcup_{i\in I} G_{i}\;\middle|\; (\forall i\in I)f(i)\in G_{i}, f(i)= 1_{G_{i}} \text{ for all but finitely many } i\right\}. \end{equation} 

The {\em weak direct product} of the groups $\{G_{i}:i\in I\}$ is the set $\prod^{weak}_{i\in I}G_{i}$ with the operation of component-wise multiplicative defined for all $f,g\in \prod^{weak}_{i\in I}G_{i}$ by $(fg)(i)=f(i)g(i)$ for all $i\in I$. A field $\mathcal{K}$ is {\em algebraically closed} if every non-constant polynomial in $\mathcal{K}[x]$ has a root in $\mathcal{K}$.
\end{defn}

\begin{defn}{\textbf{(Combinatorial definitions)}}
The {\em degree} of a vertex $v\in V_{G}$ of a graph $G=(V_{G}, E_{G})$ is the number of edges emerging from $v$.
A graph $G=(V_{G}, E_{G})$ is {\em locally finite} if every vertex of $G$ has a finite degree. We say that a graph $G=(V_{G}, E_{G})$ is {\em locally well-orderable} if for every $v \in V_{G}$, the set of neighbors of $v$ is well-orderable. Given a non-negative integer $n$, a {\em path of length $n$} in the graph $G=(V_{G}, E_{G})$ is a one-to-one finite sequence $\{x_{i}\}_{0\leq i \leq n}$ of vertices such that for each $i < n$, $\{x_{i}, x_{i+1}\} \in E_{G}$; such a path joins $x_{0}$ to $x_{n}$. The graph $G$ is {\em connected} if any two vertices are joined by a path of finite length. 
For each integer $n \geq 3$, an {\em $n$-cycle} of $G$ is a path $\{x_{i}\}_{0\leq i< n}$ such that $\{x_{n-1}, x_{0}\} \in E_{G}$ and an {\em $n$-bush} is any connected graph with no $n$-cycles. We denote by $K_{n}$ the complete graph on $n$ vertices. We denote by $C_{n}$ the circuit of length $n$.
A {\em forest} is a graph with no cycles and a {\em tree}
is a connected forest.
A {\em spanning} subgraph $H=(V_{H}, E_{H})$ of $G=(V_{G}, E_{G})$ is a subgraph that contains all the vertices of $G$ i.e., $V_{H}=V_{G}$.  A {\em complete bipartite graph} is a graph $G=(V_{G}, E_{G})$ whose vertex set $V_{G}$ can be partitioned into two subsets $V_{1}$ and $V_{2}$ such that no edge has both endpoints in the same subset, and every possible edge that could connect vertices in different subsets is a part of the graph.
A complete bipartite graph with partitions of size $\vert V_{1}\vert = m$ and $\vert V_{2}\vert = n$, is denoted by $K_{m,n}$ for any natural number $m,n$.
Let $(P, \leq)$ be a partially ordered set or a poset. 
A subset $D \subseteq P$ is called a {\em chain} if $(D, \leq\restriction D)$ is linearly ordered. A subset $A\subseteq P$ is called an {\em antichain}
if no two elements of $A$ are comparable under $\leq$. The size of the largest antichain of the poset $(P, \leq)$ is known as its {\em width}.  A subset $C \subseteq P$ is called {\em cofinal} in $P$ if for every $x \in P$ there is an element $c \in C$ such that $x \leq c$.
\end{defn}

\subsection{Permutation models.} In this subsection, we provide a brief description of the construction of Fraenkel-Mostowski permutation models of $\mathsf{ZFA}$ from \cite[Chapter 4]{Jec1973}. 
Let $M$ be a model of $\mathsf{ZFA+AC}$ where $A$ is a set of atoms or urelements. Let $\mathcal{G}$ be a group of permutations of $A$.  A set $\mathcal{F}_{1}$ of subgroups of $\mathcal{G}$ is a normal filter on $\mathcal{G}$ if for all subgroups $H, K$ of $\mathcal{G}$, the following holds.

\begin{enumerate}
    \item $\mathcal{G}\in \mathcal{F}_{1}$,
    \item If $H\in \mathcal{F}_{1}$ and $H\subseteq K$, then $K\in \mathcal{F}_{1}$,
    \item If $H\in \mathcal{F}_{1}$ and $K\in \mathcal{F}_{1}$ then $H\cap K\in \mathcal{F}_{1}$,
    \item If $\pi\in \mathcal{G}$ and $H\in \mathcal{F}_{1}$, then $\pi H \pi^{-1}\in \mathcal{F}_{1}$,
    \item For each $a \in A$, $\{\pi \in \mathcal{G} : \pi(a) = a\} \in \mathcal{F}_{1}$.  
\end{enumerate}

Let $\mathcal{F}$ be a normal filter of subgroups of $\mathcal{G}$.
For $x\in M$, we say
\begin{equation}
    sym_{\mathcal {G}}(x) =\{g\in \mathcal {G} : g(x) = x\} \text{ and } \text{fix}_{\mathcal{G}}(x) =\{\phi \in \mathcal{G} : \forall y \in x (\phi(y) = y)\}.
\end{equation}

We say $x$ is {\em symmetric} if $sym_{\mathcal{G}}(x)\in\mathcal{F}$ and $x$ is {\em hereditarily symmetric} if $x$ is symmetric and each element of the transitive closure of $x$ is symmetric. We define the permutation model $\mathcal{N}$ with respect to $\mathcal{G}$ and $\mathcal{F}$, to be the class of all hereditarily symmetric objects. It is well-known that $\mathcal{N}$ is a model of $\mathsf{ZFA}$ (cf. \cite[Theorem 4.1]{Jec1973}). 
A family $\mathcal{I}_{1}$ of subsets of $A$ is a normal ideal if the following holds.\begin{enumerate}\item $\emptyset\in \mathcal{I}_{1}$,\item  If $E \in \mathcal{I}_{1}$ and $F \subseteq E$, then $F \in \mathcal{I}_{1}$,\item  If $E,F \in \mathcal{I}_{1}$, then $E \cup F \in \mathcal{I}_{1}$,\item  If $\pi\in \mathcal{G}$ and $E \in \mathcal{I}_{1}$, then $\pi[E]\in \mathcal{I}_{1}$,\item For each $a\in A$, $\{a\}\in \mathcal{I}_{1}$.\end{enumerate}
If $\mathcal{I}\subseteq\mathcal{P}(A)$ is a normal ideal, then the set $\{$fix$_{\mathcal{G}}(E): E\in\mathcal{I}\}$ generates a normal filter (say $\mathcal{F}_{\mathcal{I}}$) over $\mathcal G$.  Let $\mathcal{N}$ be the permutation model determined by $M$, $ \mathcal{G},$ and $\mathcal{F}_{\mathcal{I}}$. We say $E\in \mathcal{I}$ is a {\em support} of a set $\sigma\in \mathcal{N}$ if fix$_{\mathcal{G}}(E)\subseteq sym_{\mathcal{G}} (\sigma$). 

\begin{lem}
{\em The following hold:
\begin{enumerate}
    \item An element $x$ of $\mathcal{N}$ is well-orderable in $\mathcal{N}$ if and only if {\em fix$_{\mathcal{G}}(x)\in \mathcal{F}_{\mathcal{I}}$} {\em (cf. \cite[Equation (4.2), p.47]{Jec1973})}. Thus, an element $x$ of $\mathcal{N}$ with support $E$ is well-orderable in $\mathcal{N}$ if {\em fix$_{\mathcal{G}}(E) \subseteq$ fix$_{\mathcal{G}}(x)$}.
    
    \item Let $\mathcal{G}$ be a group of permutations of a set of atoms $A$ and let $\mathcal{I}$ be a normal ideal of supports. Let $\mathcal{V}$ be the permutation model given by $\mathcal{G}$ and $\mathcal{I}$. Then for all $\pi \in \mathcal{G}$ and all $x \in \mathcal{V}$ such that $E$ is a support of $x$, {\em $sym_{\mathcal{G}}(\pi x) = \pi$ $sym_{\mathcal{G}}(x)\pi^{-1}$} 
    and {\em fix$_{\mathcal{G}}(\pi E) = \pi$ fix$_{\mathcal{G}}(E)\pi^{-1}$}
    {\em (cf. \cite[proof of Lemma 4.4]{Jec1973})}.
\end{enumerate}
}
\end{lem}

A {\em pure set} in a model $M$ of $\mathsf{ZFA}$ is a set with no atoms in its transitive closure.
The {\em kernel} is the class of all pure sets of $M$. In this paper, 
\begin{itemize}
    \item Fix an integer $n\geq 2$. We denote by $\mathcal{N}_{HT}^{1}(n)$ the permutation model constructed in \cite[Theorem 8]{HT2020}.
    \item We denote by $\mathcal{N}_{1}$ the basic Fraenkel model (cf. \cite{HR1998}).
    
    \item We denote by $\mathcal{N}_{12}(\aleph_{1})$, the following variant of the basic Fraenkel model: Let $A$ be a set of atoms of size $\aleph_{1}$, $\mathcal{G}$ be the group of all permutations of $A$, and the supports are countable subsets of $A$ (cf. \cite{HR1998}).
    
    \item We denote by $\mathcal{N}_{12}(\aleph_{\alpha})$, the following variant of the basic Fraenkel model: Let $A$ be a set of atoms of size $\aleph_{\alpha}$, $\mathcal{G}$ be the group of all permutations of $A$, and the supports are subsets of $A$ with cardinality less than $\aleph_{\alpha}$ (cf. \cite{HR1998}).
    
    \item We denote by $\mathcal{V}_{p}$ the finite partition model constructed  in \cite{Bru2016}.
    \item We denote by $\mathcal{V}_{p}^{+}$ the countable partition model mentioned in \cite[section 5]{Bru2016}.
    \item We denote by $\mathcal{N}_{6}$ L\'{e}vy’s permutation model (cf. \cite{HR1998}).
\end{itemize}

%%%%%%%%%%%%%%%%%%%%%%
\section{\textbf{Form 269}, \textbf{Form 233}, and \textbf{Form 304}}
We recall that $\mathsf{AC^{LO}}$ implies $\mathsf{LW}$ in $\mathsf{ZFA}$ and that the implication is not reversible in $\mathsf{ZFA}$ (cf. \cite{HR1998}).

\begin{thm}
{\em $\mathsf{AC^{LO}}$ does not imply Form 269 in $\mathsf{ZFA}$. So, neither $\mathsf{LW}$ nor $\mathsf{AC^{WO}}$ implies Form 269 in $\mathsf{ZFA}$.}
\end{thm}

\begin{proof}
We present two known models.

{\em First model}: Fix a successor aleph $\aleph_{\alpha +1}$. In the proof of \cite[Theorem 8.9]{Jec1973}, Jech proved that $\mathsf{AC^{WO}}$ holds in the permutation model $\mathcal{N}_{12}(\aleph_{\alpha+1})$.

We recall a variant of $\mathcal{N}_{12}(\aleph_{\alpha+1})$ from \cite[Theorem 3.5(i)]{Tac2019}. 
Let $\mathcal{N}$ be the permutation model, in which $A$ is a set of atoms of size $\aleph_{\alpha+1}$, $\mathcal{G}'$ is the group of all permutations of $A$ which move at most $\aleph_{\alpha}$ atoms, and the supports are subsets of $A$ with cardinality less than $\aleph_{\alpha+1}$.
Tachtsis \cite[Theorem 3.5(i)]{Tac2019} proved that $\mathcal{N}=\mathcal{N}_{12}(\aleph_{\alpha+1})$ and $\mathsf{AC^{LO}}$ holds in $\mathcal{N}$. We slightly modify the arguments of \cite[Note 91]{HR1998} to prove that Form 269 fails in $\mathcal{N}$. We show that for any set $X$ in $\mathcal{N}$ if the set $[X]^{2}$ of two-element subsets of $X$ has a choice function, then $X$ is well-orderable in $\mathcal{N}$. Assume that $X$ is such a set and let $E$ be a support of $X$ and of a choice function $f$ on $[X]^{2}$. In order to show that $X$ is well-orderable in $\mathcal{N}$, it is enough to prove that
fix$_{\mathcal{G}'}(E)$ $\subseteq$ fix$_{\mathcal{G}'}(X)$ (cf. Lemma 2.4(1)).  Assume fix$_{\mathcal{G}'}(E)$ $\nsubseteq$ fix$_{\mathcal{G}'}(X)$, then there is a $y \in X$ and a $\phi \in$ fix$_{\mathcal{G}'}(E)$ with $\phi(y) \neq y $. Under such assumptions, Tachtsis constructed a permutation $\psi \in$ fix$_{\mathcal{G}'}(E)$ such
that $\psi(y) \neq y$ but $\psi^{2}(y)=y$ (cf. the proof of $\mathsf{LW}$ in $\mathcal{N}$ from \cite[Theorem 3.5(i)]{Tac2019}). This contradicts our choice of $E$ as a support for a choice function on $[X]^{2}$ since $\psi$ fixes $ \{\psi(y),y\}$ but moves both of its elements. So Form 269 fails in $\mathcal{N}$.

{\em Second model}: We consider the permutation model $\mathcal{N}$ given in the proof of \cite[Theorem 4.7]{Tac2019a} where $\mathsf{AC^{LO}}$ and therefore $\mathsf{LW}$  hold. Following the above arguments and the arguments in \cite[claim 4.10]{Tac2019a}, we can see that Form 269 fails in $\mathcal{N}$.
\end{proof}

We recall a result of Pincus, which we need in order to prove Theorem 3.3.

\begin{lem}{(\textbf{Pincus; \cite[Note 41]{HR1998}})} {\em If $\mathcal{K}$ is an algebraically closed field, if $\pi$ is a non-trivial automorphism of $\mathcal{K}$ satisfying $\pi^{2}= $ {\em id}$_{\mathcal{K}}$, and if $i \in\mathcal{K}$ is a square root of $-1$, then $\pi(i) = -i \neq i$.}
\end{lem}

\begin{thm}
{\em Fix any $2\leq n\in \omega$. There is a model $\mathcal{M}$ of $\mathsf{ZFA}$ where  $\mathsf{AC_{n}^{-}}$ and the statement `there are no amorphous sets' fail.
Moreover, the following hold in $\mathcal{M}$:
\begin{enumerate}
    \item {\em Form 269} fails.
    \item {\em Form 233} holds.
    \item {\em Form 304} holds.
\end{enumerate}
}
\end{thm}

\begin{proof}
We consider the permutation model constructed by Halbeisen--Tachtsis \cite[Theorem 8]{HT2020} where for arbitrary integer $n\geq 2$, $\mathsf{AC_{n}^{-}}$ fails. We fix an arbitrary integer $n\geq 2$ and recall the model constructed in the proof of \cite[Theorem 8]{HT2020}.
We start with a model $M$ of $\mathsf{ZFA+AC}$ where $A$ is a countably infinite set of atoms written as a disjoint union $\bigcup\{A_{i}:i\in \omega\}$ where for each $i\in \omega$, $A_{i}=\{a_{i_{1}}, a_{i_{2}},..., a_{i_{n}}\}$ and $\vert A_{i}\vert = n$. The group $\mathcal{G}$ is defined in \cite{HT2020} in a way so that if $\eta\in \mathcal{G}$, then $\eta$ only moves finitely many atoms and for all $i\in \omega$, $\eta(A_{i})=A_{k}$ for some $k\in \omega$. Let $\mathcal{F}$ be the filter of subgroups of $\mathcal{G}$ generated by $\{$fix$_{\mathcal{G}}(E): E\in [A]^{<\omega}\}$. 
We denote by $\mathcal{N}_{HT}^{1}(n)$
the Fraenkel–Mostowski  permutation  model determined by $M$, $\mathcal{G}$, and $\mathcal{F}$. Following point 1 in the proof of \cite[Theorem 8]{HT2020}, both $A$ and $\mathcal{A}=\{A_{i}\}_{i\in \omega}$ are amorphous in $\mathcal{N}_{HT}^{1}(n)$. If $X$ is a set in $\mathcal{N}_{HT}^{1}(n)$, then without loss of generality we may assume that $E=\bigcup_{i=0}^{m} A_{i}$ is a support of $X$ for some $m\in\omega$.

\begin{claim}
{\em Suppose $X$ is not a well-ordered set in $\mathcal{N}_{HT}^{1}(n)$, and let $E=\bigcup_{i=0}^{m} A_{i}$ be a support of $X$  for some $m\in\omega$. Then there is a $t\in X$ with support $F\supsetneq E$ and a permutation $\delta\in$ {\em fix$_{\mathcal{G}}(E)$} such that $\delta^{2}$ is the identity, and $\delta(t)\neq t$.

%and an element $y \in X$ $t \neq y$, $\delta(t) = y$ and $\delta(y) = t$. In particular, $\delta^{2}$ is the identity, and $\delta(t)\neq t$.
}
\end{claim}

\begin{proof}
Since $X$ is not well-ordered, and $E$ is a support of $X$,
fix$_{\mathcal{G}}(E)\nsubseteq$ fix$_{\mathcal{G}} (X)$ by Lemma 2.4(1). So there is a $t \in X$ and a $\psi \in$ fix$_{\mathcal{G}}(E)$ such that $\psi(t) \neq t$. Let $F$ be a support of $t$ containing $E$. Without loss of generality, we may assume that $F$ is a union of finitely many $A_{i}$'s. We sligtly modify the arguments of \cite[claim 4.10]{Tac2019a}.
Let $W = \{a \in A : \psi(a) \neq a\}$. We note that $W$ is finite since if $\eta\in \mathcal{G}$, then $\eta$ only moves finitely many atoms. Let $U$ be a finite subset of $A$ which is disjoint from $F \cup W$ and such that there exists a bijection $H : tr(U) \rightarrow tr((F \cup W)\backslash E)$ (where for a set $x \subseteq A$, $tr(x) = \{i \in \omega : A_{i} \cap x \neq \emptyset\}$) with the property that if $i \in tr((F \cup W)\backslash E)$ is such that $A_{i} \subseteq (F \cup W)\backslash E$ then $A_{H^{-1}(i)} \subseteq U$; otherwise if $A_{i} 
\nsubseteq (F \cup W)\backslash E$,
which means that $A_{i} \cap F = \emptyset$ and $A_{i} \nsubseteq W$, then $\vert W \cap A_{i}\vert =\vert U \cap A_{H^{-1}(i)}\vert$. Let $f : U \rightarrow (F \cup W)\backslash E$ be a bijection such that $\forall i \in tr(U)$, $f \restriction U \cap A_{i}$ is a one-to-one function from $U \cap A_{i}$ onto $((F \cup W)\backslash E) \cap A_{H(i)}$. Let $f' : \bigcup_{i\in tr(U)} A_{i}\backslash (U \cap A_{i}) \rightarrow \bigcup_{i\in tr(U)} A_{H(i)}\backslash(((F \cup W)\backslash E) \cap A_{H(i)})$ be a bijection such that $\forall i \in tr(U)$, $f' \restriction (A_{i}\backslash (U \cap A_{i}))$ is a one-to-one function from $A_{i}\backslash (U \cap A_{i})$ onto $A_{H(i)}\backslash(((F \cup W)\backslash E) \cap A_{H(i)})$. Let 

\begin{center}$\delta = \prod_{u\in U} (u, f (u))\prod_{u\in \bigcup_{i\in tr(U)} A_{i}\backslash (U \cap A_{i})} (u, f' (u))$
\end{center}

be a product of disjoint transpositions. It is clear that $\delta$ only moves finitely many atoms, and for all $i\in \omega$, $\delta(A_{i})=A_{k}$ for some $k\in \omega$. Moreover, $\delta \in$ fix$_{\mathcal{G}}(E)$, $\delta^{2}$ is the identity, and $\delta(t)\neq t$ by the arguments in \cite[claim 4.10]{Tac2019a}. 
\end{proof}

\begin{claim}
{\em In $\mathcal{N}_{HT}^{1}(n)$, the following hold:
\begin{enumerate}
    \item {\em Form 269} fails.
    \item {\em Form 304} holds.
    \item {\em Form 233} holds.
\end{enumerate}
}
\end{claim}
\begin{proof}
(1). Following claim 3.4 and the arguments in the proof of Theorem 3.1, Form 269 fails in $\mathcal{N}_{HT}^{1}(n)$.
    
(2). We modify the arguments of \cite[Note 116]{HR1998} to prove that Form 304 holds in $\mathcal{N}_{HT}^{1}(n)$. Let $X$ be an infinite Hausdorff space in $\mathcal{N}_{HT}^{1}(n)$, and $E=\bigcup_{i\in K} A_{i}$ be a support of $X$ and its topology where $K\in [\omega]^{<\omega}$. We show there is an infinite $Y \subseteq X$ in $\mathcal{N}_{HT}^{1}(n)$ such that $Y$ has no infinite compact subsets in $\mathcal{N}_{HT}^{1}(n)$. If $X$ is well-orderable, then we can use transfinite induction without using any form of choice to finish the proof. Suppose $X$ is not well-orderable in $\mathcal{N}_{HT}^{1}(n)$.
By Lemma 2.4(1), there is an $x \in X$ and a $\phi \in \text{fix}_{\mathcal{G}}(E)$ such that $\phi(x) \neq x$. Let $F=\bigcup_{i\in K'} A_{i}$ be a support of $x$ where $K'\in [\omega]^{<\omega}$. 
Since $E$ is not a support of $x$, $F\backslash E\neq \emptyset$. Without loss of generality assume that $E\subsetneq F$. We also assume that $\{A_{i}:i\in K'\}$ has the fewest possible copies $A_{j}$ outside $\{A_{i}:i\in K\}$. Let $i_{0} \in K'$ such that $A_{i_{0}} \cap E = \emptyset$. We define
\begin{center}
    $f = \{(\psi(x),\psi(A_{i_{0}})) : \psi\in$ fix$_{\mathcal{G}}(F\backslash A_{i_{0}})\}$.
\end{center}

Tacthsis proved that $f$ is a function with $dom(f) \subseteq X$ and $ran(f) = \mathcal{A}\backslash\{A_{i} : i \in K', i \neq i_{0}\}$, where $\mathcal{A} = \{A_{i} : i\in \omega\}$ and $Y=dom(f)$ is an amorphous subset of $X$ (cf. proof of \cite[Lemma 2]{Tac2016a}).
Since $\phi(x)\neq x$ and $X$ is an infinite Hausdorff space, we can choose open sets $C$ and $D$ so that $x \in C$, $\phi(x) \in D$ and $C \cap D = \emptyset$. Since $Y$ is amorphous in $\mathcal{N}_{HT}^{1}(n)$, every subset of $Y$ in the model must be finite or cofinite. Thus at least one of $Y \cap C$ or $Y \cap D$ is finite. We may assume that $Y \cap C$ is finite. Then we can conclude that 
\begin{center}
$\mathcal{C} = \{\psi(C)\cap Y : \psi\in$ fix$_{\mathcal{G}}(F\backslash A_{i_{0}})\}$     
\end{center}
is an open cover for $Y$ and each element of $\mathcal{C}$ is finite. So there is an infinite $Y \subseteq X$ in $\mathcal{N}_{HT}^{1}(n)$ such that for any infinite subset $Z$ of $Y$, $\mathcal{C}$ is an open cover for $Z$ without a finite subcover. 

(3). We follow the arguments due to Pincus from \cite[Note 41]{HR1998} and use claim 3.4 to prove that Form 233 holds in $\mathcal{N}_{HT}^{1}(n)$. 
For the reader's convenience, we write down the proof. Let $(\mathcal{K}, +, \cdot,0,1)$ be a field in $\mathcal{N}_{HT}^{1}(n)$ with finite support $E \subset A$ and assume that $\mathcal{K}$ is algebraically closed. Without loss of generality assume that $E = \bigcup_{i=0}^{m} A_{i}$ for some natural number $m\in \omega$. We show that every element of $\mathcal{K}$ has support $E$ which implies that $\mathcal{K}$ is well-orderable in $\mathcal{N}_{HT}^{1}(n)$ and therefore the standard proof of the uniqueness of algebraic closures (using $\mathsf{AC}$) is valid in $\mathcal{N}_{HT}^{1}(n)$. For the sake of contradiction, assume that $x\in\mathcal{K}$ does not have support $E$. 
By claim 3.4, there is a permutation $\psi\in$ fix$_{\mathcal{G}}(E)$ such that $\psi(x)\neq x$ and $\psi^{2}$ is the identity. The permutation $\psi$ induces an automorphism of $(\mathcal{K}, +, \cdot,0,1)$ and we can therefore apply Lemma 3.2 to conclude that $\psi(i) = - i \neq i$ for some square root $i$ of $-1$ in $\mathcal{K}$. 

On the other hand, we can follow the arguments from \cite[Note 41]{HR1998} to see that for every permutation $\pi$ of $A$ such that $\pi\in$ fix$_{\mathcal{G}}(E)$, $\pi(i)=i$ for every square root $i$ of $-1$ in $\mathcal{K}$. In particular, fix an $i=\sqrt{-1} \in\mathcal{K}$. It is enough to show that $E$ is a support of $i$. We note that $i$ is a solution to the equation $x^{2} + 1 = 0$ all of whose coefficients are fixed by any $\eta\in$ fix$_{\mathcal{G}}(E)$. So if $\eta\in$ fix$_{\mathcal{G}}(E)$, then $\eta(i)$ is also a solution to $x^{2}+ 1 = 0$. Suppose $E$ is not a support of $i$. Let $E'=\bigcup_{i=0}^{m+k} A_{i}$ be a support of $i$ for some natural number $k$ and let $F = E' \backslash E$. Then $F \neq \emptyset$ (since $E$ is not a support of $i$) and $F \cap E = \emptyset$.
By applying \cite[Lemma 1]{Tac2016a} (where Tachtsis proved that if $x \in \mathcal{N}_{HT}^{1}(n)$ and $E_{1}, E_{2}$ are supports of $x$, then $E_{1}\cap E_{2}$ is a support of $x$)\footnote{We note that Tachtsis assumed that a support $E$ has the property that $\forall i\in\omega, A_{i} \subseteq E$ or $A_{i} \cap E = \emptyset$ in order to prove \cite[\textbf{Lemma 1}]{Tac2016a}.}, 
we can see that if $\phi, \phi'\in$ fix$_{\mathcal{G}}(E)$ and $\phi(F) \cap \phi'(F) = \emptyset$, then $\phi(i) \neq \phi'(i)$. Consequently, we can obtain an infinite set $S = \{\phi_{k}(i) : k \in \omega\}$ such that $\phi_{k} \in$ fix$_{\mathcal{G}}(E)$ and, $\phi_{k}(i)$ is in $\mathcal{N}_{HT}^{1}(n)$ for every $k \in \omega$, and for all $k,l \in \omega$, if $k \neq l$ then $\phi_{k}(i) \neq \phi_{l}(i)$. Thus, the equation $x^{2} + 1 = 0$ has 
infinitely many solutions in $\mathcal{K}$, which is a contradiction. 
Thus, $E$ is a support of $i$. This completes the proof of the third assertion.
\end{proof}
\end{proof}

\begin{remark} 
Fix $2\leq n\in\omega$. For each regular $\aleph_{\alpha}$,\footnote{We assume that $\alpha$ is an
ordinal and that $\aleph_{\alpha}$ is the $\alpha^{th}$ infinite initial ordinal (where an ordinal $\beta$ is “initial” if $\beta$ is not equipotent with an ordinal $\gamma \in \beta$).} we denote by $\mathsf{CAC^{\aleph_{\alpha}}_{1}}$ the statement {\em `If in a poset all antichains are finite and all chains have size at most $\aleph_{\alpha}$ and
there exists at least one chain with size $\aleph_{\alpha}$ then the poset has size $\aleph_{\alpha}$'}. In \cite[Theorem 4.3, Remark 4.4]{Ban2} we proved that the statement “For every regular $\aleph_{\alpha}$, $\mathsf{CAC^{\aleph_{\alpha}}_{1}}$” holds in $\mathcal{N}_{HT}^{1}(n)$ and $\mathcal{N}_{1}$. 
We present different proofs to show that the statement “For every regular $\aleph_{\alpha}$, $\mathsf{CAC^{\aleph_{\alpha}}_{1}}$” holds in $\mathcal{N}_{HT}^{1}(n)$ and $\mathcal{N}_{1}$. First, we recall the following result communicated to us by Tachtsis from \cite{Ban2}.

\begin{lem}(cf. \cite[Lemma 4.1, Corollary 4.2]{Ban2})
The following hold:
\begin{enumerate}
    \item The statement {\em `If $(P, \leq)$ is a poset such that P is well-ordered, and if all antichains in P are finite and all chains in P are countable, then P is countable'} holds in any Fraenkel-Mostowski model.
    \item $\mathsf{UT(\aleph_{\alpha},\aleph_{\alpha},\aleph_{\alpha})}$ implies the statement {\em `If $(P, \leq)$ is a poset such that P is well-ordered, and if all antichains in P are finite and all chains in P have size at most $\aleph_{\alpha}$ and there
exists at least one chain with size $\aleph_{\alpha}$, then $P$ has size $\aleph_{\alpha}$'} for any regular $\aleph_{\alpha}$ in $\mathsf{ZF}$.
\end{enumerate}
\end{lem}

Fix $\mathcal{N}\in \{\mathcal{N}_{HT}^{1}(n), \mathcal{N}_{1}\}$. Let $(P,\leq)$ be a poset in $\mathcal{N}$ such that all antichains in $P$ are finite and all chains in $P$ have size $\aleph_{\alpha}$. Let $E\in [A]^{<\omega}$ be a support of $(P,\leq)$.

\textbf{Case (i):} Let $\mathcal{N}= \mathcal{N}_{HT}^{1}(n)$. Then for each element $p\in P$, either $\text{Orb}_{\text{fix}_{\mathcal{G}}(E)}(p) = \{\phi(p): \phi\in$ fix$_{\mathcal{G}}(E)\} = \{p\}$ or
$\text{Orb}_{\text{fix}_{\mathcal{G}}(E)}(p)$ is infinite (cf. \cite[Remark 2.2]{Tac2016a}). Following the arguments of \cite[claim 3]{Tac2016a} we can see that for each $p \in P$, $\text{Orb}_{\text{fix}_{\mathcal{G}}(E)}(p)$ is an anti-chain in $P$. So by assumption, $\text{Orb}_{\text{fix}_{\mathcal{G}}(E)}(p) = \{p\}$. Following the arguments of \cite[claim 4]{Tac2016a} we can see that   $\mathcal{O}=\{\text{Orb}_{\text{fix}_{\mathcal{G}}(E)}(p): p\in P\}$ is a well-ordered partition of $P$. Thus $P$ is also well-orderable. The rest follows from Lemma 3.7(2), since $\mathsf{UT(WO,WO,WO)}$ holds in $\mathcal{N}$ and $\mathsf{UT(WO,WO,WO)}$ implies $\mathsf{UT(\aleph_{\alpha},\aleph_{\alpha},\aleph_{\alpha})}$ in any FM-model (cf. \cite[p. 176]{HR1998}).

\textbf{Case (ii):} Let $\mathcal{N}= \mathcal{N}_{1}$. If $P$ is well-orderable, then we are done. Suppose $P$ is not well-orderable. By Lemma 2.4(1), there is a $t \in P$ and a $\pi \in \text{fix}_{\mathcal{G}}(E)$ such that $\pi(t) \neq t$. Let $F \cup \{a\}$ be a support of $t$ where $a\in A\backslash (E\cup F)$. Under such assumptions, Blass \cite[p.389]{Bla1977} proved that

\begin{center}
    $f=\{(\pi(a),\pi(t)):\pi\in \text{fix}_{\mathcal{G}}(E\cup F)\}$
\end{center}

is a bijection from $A \backslash (E \cup F)$ onto $ran(f)\subset P$. Now, $ran(f)=\text{Orb}_{\text{fix}_{\mathcal{G}}(E\cup F)}(t)=\{\pi(t):\pi\in \text{fix}_{\mathcal{G}}(E\cup F)\}$ is an infinite antichain of $P$ (cf. the proof of \cite[Lemma 9.3]{Jec1973}) in $\mathcal{N}$, which contradicts our assumption. 

Fix $k,n\in \omega\backslash\{0,1\}$. Tachtsis \cite[Theorem 3.7, Remark 3.8]{Tac2019b} proved that the statement ``If $P$ is a poset with width $k$ while at least one $k$-element subset of $P$ is an antichain, then $P$ can be partitioned into $k$ chains'', abbreviated as $\mathsf{DT}$, holds in $\mathcal{N}_{1}$ and $\mathcal{N}_{HT}^{1}(2)$. Using the above arguments we can give a different proof of $\mathsf{DT}$ in $\mathcal{N}_{HT}^{1}(n)$ and $\mathcal{N}_{1}$ since $\mathsf{DT}$ for well-ordered infinite posets with finite width is provable in $\mathsf{ZF}$ \cite[Theorem 3.1(i)]{Tac2019b}. 
\end{remark}
%%%%%%%%%%%%%%%%%%%%%
\section{Partition models, weak choice forms, and permutations of infinite sets}
We recall some known results, which we need in order to prove Theorem 4.5 and Theorem 4.8.
\begin{lem}(Keremedis--Herrlich--Tachtsis; cf. \cite[Remark 2.7]{Tac2016}, \cite[Theorem 3.1]{KH1962})
{\em The following hold:
\begin{enumerate}
    \item $\mathsf{AC_{fin}^{\omega} + MA(\aleph_{0})}\rightarrow$ `for every infinite set $X$, $2^{X}$ is Baire'.
    \item `For every infinite set $X$, $2^{X}$ is Baire' $\rightarrow$ `For every infinite set $X$, $\mathcal{P}(X)$ is Dedekind-infinite'.
\end{enumerate}
}
\end{lem}

\begin{lem}(L\'{e}vy; \cite{Lev1962})
{\em $\mathsf{MC}$ if and only if every infinite set has a well-ordered partition into non-empty finite sets.}
\end{lem}

\begin{lem}(Howard--Saveliev--Tachtsis; \cite[Lemma 1.3, Theorem 3.1]{HST2016})
{\em The following hold:
\begin{enumerate}
    \item $\mathsf{\leq\aleph_{0}}$-$\mathsf{MC}$ if and only if every infinite set has a well-ordered partition into non-empty countable sets.
    \item $\mathsf{\leq\aleph_{0}}$-$\mathsf{MC}$ implies ``for every infinite set $X$, $\mathcal{P}(X)$ is Dedekind-infinite”, which in turn is equivalent
to ``for every infinite set $P$ there is a partial ordering $\leq$ on $P$ such that $(P,\leq)$ has a countably infinite disjoint family of cofinal subsets".
\end{enumerate}
}
\end{lem}

\begin{lem}{(Tachtsis; \cite[Theorem 3.1]{Tac2019})} 
{\em The following hold:
\begin{enumerate}
    \item Each of the following statements implies the one beneath it:
\begin{enumerate}
    \item {\em Form 3};
    \item $\mathsf{ISAE}$;
    \item $\mathsf{EPWFP}$;
    \item For every infinite set $X$, Sym($X$) $\neq$ FSym($X$).
\end{enumerate}
\item $\mathsf{DF = F}$ implies ``For every infinite set $X$, Sym($X$) $\neq$ FSym($X$)".
\end{enumerate}
}
\end{lem}

\subsection{Weak choice forms in the finite partition model} We recall the finite partition model $\mathcal{V}_{p}$ from \cite{Bru2016}. In order to describe $\mathcal{V}_{p}$, we start with a model $M$ of $\mathsf{ZFA+AC}$ where $A$ is a  countably infinite set of atoms. Let $\mathcal{G}$ be the group of all permutations of $A$, $S$ be the set of all finite partitions of $A$, and $\mathcal{F}$ = $\{H:$ $H$ is a subgroup of $\mathcal{G}$, $H \supseteq$ fix$_{\mathcal{G}}(P)$ for some $P \in S\}$ be a normal filter of subgroups of $\mathcal{G}$. The model $\mathcal{V}_{p}$ is the permutation model determined by $M$, $\mathcal{G}$ and $\mathcal{F}$. In $\mathcal{V}_{p}$, $A$ has no infinite amorphous subset (cf. \cite[Proposition 4.3]{Bru2016}).

\begin{thm}
{\em The following hold in $\mathcal{V}_{p}$:
\begin{enumerate}
    \item If $X\in \{${\em Form 3}$, \mathsf{ISAE}, \mathsf{EPWFP}\}$, then $X$ fails.
    \item $\mathsf{AC_{n}}$ fails for any integer $n\geq 2$.
    \item $\mathsf{MA(\aleph_{0})}$ fails.
    \item If $X\in \{\mathsf{MC}, \mathsf{\leq\aleph_{0}}$-$\mathsf{MC}\}$, then $X$ fails.
\end{enumerate}
}
\end{thm}

\begin{proof}
(1). 
By Lemma 4.4, it is enough to show that
(Sym($A$))$^{\mathcal{V}_{p}}$ = FSym($A$).
For the sake of contradiction, assume that
$f$ is a permutation of $A$ in $\mathcal{V}_{p}$, which moves infinitely many atoms.
Let $P =\{P_{j}:j\leq k\}$ be a support of $f$ for some $k\in\omega$.
Without loss of generality, assume that $P_{0},..., P_{n}$ are the singleton and tuple blocks for some $n< k$.
Then there exist $n< i \leq k$ where $a \in P_{i}$ and $b \in \bigcup P\backslash (P_{0}\cup...\cup P_{n} \cup \{a\})$ such that $b = f(a)$. 

\textbf{Case (i):}
Let $b \in P_{i}$. Consider $\phi\in$ fix$_{\mathcal{G}}(P)$ such that $\phi$
fixes all the atoms in all the blocks other than $P_{i}$ and $\phi$ moves every atom in $P_{i}$ except $b$. Thus, $\phi(b)=b$, $\phi(a)\neq a$, and $\phi(f) = f$ since $P$ is the support of $f$. Thus $(a,b)\in f\implies (\phi(a), \phi(b))\in \phi(f)\implies (\phi(a), b)\in f$. So $f$ is not injective; a contradiction. 

\textbf{Case (ii):} Let $b \not\in P_{i}$. Consider $\phi\in$ fix$_{\mathcal{G}}(P)$ such that $\phi$
fixes all the atoms in all the blocks other than $P_{i}$ and $\phi$ moves every atom in $P_{i}$. Then again we can obtain a contradiction as in Case (i).

(2). Fix any integer $n\geq 2$. We show that the set $[A]^{n}$ has no choice function in $\mathcal{V}_{p}$. Assume that $f$ is a choice function of $[A]^{n}$ and let $P$ be a support of $f$.
Since $A$ is countably infinite and $P$ is a finite partition of $A$, there is a $p \in P$ such that $\vert p\vert$ is infinite. Let $a_{1}, a_{2},...,a_{n} \in p$ and $\pi \in $ fix$_{\mathcal{G}}(P)$ be such that $\pi a_{1} = a_{2}$, $\pi a_{2} = a_{3}$,..., $\pi a_{n-1} = a_{n}$, $\pi a_{n} = a_{1}$. Without loss of generality, we assume that $f(\{a_{1}, a_{2},...,a_{n}\})= a_{1}$. Thus,
$\pi f(\{a_{1}, a_{2},...,a_{n}\})= \pi a_{1} \implies f(\{\pi(a_{1}), \pi(a_{2}),...,\pi(a_{n})\})= a_{2} \implies
    f(\{a_{2}, a_{3},...,a_{n}, a_{1}\})= a_{2}.$
Thus $f$ is not a function; a contradiction.

(3). It is known that $\mathcal{P}(A)$ is Dedekind-finite  and $\mathsf{UT(WO,WO,WO)}$ holds in $\mathcal{V}_{p}$ (cf. \cite[Proposition 4.9, Theorem 4.18]{Bru2016}). So $\mathsf{AC_{fin}^{\omega}}$ holds as well. Thus by Lemma 4.1(2), the
statement “for every infinite set $X$, $2^{X}$ is Baire” is false in $\mathcal{V}_{p}$. Hence by Lemma 4.1(1), $\mathsf{MA(\aleph_{0})}$ is false in $\mathcal{V}_{p}$.

(4). Follows from Lemmas 4.2, 4.3(1) and the fact that $\mathsf{UT(WO,WO,WO)}$ holds in $\mathcal{V}_{p}$. Alternatively, we can also use Lemma 4.3(2), to see that $\mathsf{\leq\aleph_{0}}$-$\mathsf{MC}$ fails in $\mathcal{V}_{p}$ since $\mathcal{P}(A)$ is Dedekind-finite in $\mathcal{V}_{p}$. So we may also conclude by Lemma 4.3(2) that the statement “for every infinite set $P$ there is a partial ordering $\leq$ on $P$ such that $(P, \leq)$ has a countably infinite disjoint family of cofinal subsets" fails in $\mathcal{V}_{p}$.
\end{proof}

\subsection{Weak choice forms in the countable partition model}

Let $M$ be a model of $\mathsf{ZFA+AC}$ where $A$ is an uncountable set of atoms and $\mathcal{G}$ is the group of all permutations of $A$. 

\begin{lem}
{\em Let $S$ be the set of all countable partitions of $A$. Then $\mathcal{F}$ = $\{H:$ $H$ is a subgroup of  $\mathcal{G}$, $H \supseteq$ {\em fix$_{\mathcal{G}}(P)$} for some $P \in S\}$ is a normal filter of subgroups of $\mathcal{G}$.}
\end{lem}

\begin{proof}
We modify the arguments of \cite[Lemma 4.1]{Bru2016} and verify the clauses 1-5 of a normal filter (cf. section 2.1).
\begin{enumerate}
    \item We can see that $\mathcal{G}\in \mathcal{F}$.
    \item Let $H \in \mathcal{F}$ and $K$ be a subgroup of $\mathcal{G}$ such that $H \subseteq K$. Then there exists $P \in S$ such that fix$_{\mathcal{G}}(P) \subseteq H$. So, fix$_{\mathcal{G}}(P) \subseteq K$ and $K \in \mathcal{F}$.
    \item Let $K_{1}, K_{2}  \in \mathcal{F}$. Then there exist $P_{1}, P_{2} \in S$ such that fix$_{\mathcal{G}}(P_{1})\subseteq K_{1}$ and fix$_{\mathcal{G}}(P_{2})\subseteq K_{2}$. Let $P_{1} \land P_{2}$ denote the coarsest common refinement of $P_{1}$ and $P_{2}$, given by $P_{1} \land P_{2} = \{p \cap q : p\in P_{1}, q \in P_{2}, p \cap q\neq \emptyset\}$. Clearly, fix$_{\mathcal{G}}(P_{1} \land P_{2})\subseteq$  fix$_{\mathcal{G}}(P_{1})$ $\cap$ fix$_{\mathcal{G}}(P_{2})$ $\subseteq$ $K_{1} \cap K_{2}$. Since the product of two countable sets is countable, $P_{1} \land P_{2}\in S$. Thus $K_{1} \cap K_{2} \in \mathcal{F}$.
    \item Let $\pi \in \mathcal{G}$ and $H \in \mathcal{F}$. Then there exists $P \in S$ such that fix$_{\mathcal{G}}(P) \subseteq H$. Since fix$_{\mathcal{G}}(\pi P)$ = $\pi$ fix$_{\mathcal{G}}(P)\pi^{-1} \subseteq \pi H \pi^{-1}$ by Lemma 2.4(2), it is enough to show $\pi P \in S$.  Clearly, $\pi P$ is countable, since $P$ is countable. Following the arguments of \cite[Lemma 4.1(iv)]{Bru2016} we can see that $\pi P$ is a partition of $A$.
    \item Fix any $a \in A$. Consider any countable partition $P$ of $A$ where $\{a\}$ is a singleton block of $P$. We can see that fix$_{\mathcal{G}}(P)\subseteq\{\pi \in \mathcal{G} :\pi (a) = a\}$. Thus, $\{\pi \in \mathcal{G} :\pi (a) = a\}\in \mathcal{F}$.
\end{enumerate}
\end{proof}

We call the permutation model (denoted by $\mathcal{V}^{+}_{p}$) determined by $M$, $\mathcal{G}$, and $\mathcal{F}$, the countable partition model. Tachtsis \cite[Theorem 3.1(2)]{Tac2019} proved that $\mathsf{DF = F}$ implies ``For every infinite set $X$, Sym(X) $\neq$ FSym(X)'' in $\mathsf{ZF}$. Inspired by that idea we first prove the following.

\begin{prop}
{($\mathsf{ZF}$)} {\em The following hold:
\begin{enumerate}
    \item $\mathsf{W_{\aleph_{\alpha+1}}}$ implies `for any set $X$ of size $\not\leq\aleph_{\alpha}$, Sym(X) $\neq$ $\aleph_{\alpha}$Sym(X)'.
    \item $\mathsf{EPWFP}$ implies `for any set $X$ of size $\not\leq\aleph_{\alpha}$, Sym(X) $\neq$ $\aleph_{\alpha}$Sym(X)'.
\end{enumerate}
}
\end{prop}

\begin{proof}
(1). Let $X$ be a set of size $\not\leq\aleph_{\alpha}$ and let us assume Sym$(X)$ = $\aleph_{\alpha}$Sym$(X)$. We prove that there is no injection $f$ from $\aleph_{\alpha+1}$ into $X$. Assume there exists such an $f$. Let $\{y_{n}\}_{n\in \aleph_{\alpha+1}}$ be an enumeration of the elements of $Y=f[\aleph_{\alpha+1}]$. 
We can use transfinite recursion, without using any form of choice, to construct a bijection $h: Y\rightarrow Y$ such that $h(x)\neq x$ for any $x\in Y$. Define $g:X \rightarrow X$ as follows: $g(x)=h(x)$ if $x\in Y$, and $g(x)=x$ if $x\in X\backslash Y$. Clearly $g \in$ Sym$(X)\backslash$ $\aleph_{\alpha}$Sym$(X)$, and hence Sym$(X)$ $\neq $ $\aleph_{\alpha}$Sym$(X)$, a contradiction.

(2). This is straightforward.
\end{proof}

\begin{thm}
{\em The following hold:
\begin{enumerate}
    \item $\mathcal{N}_{12}(\aleph_{1}) \subset \mathcal{V}^{+}_{p}$.
    \item if $X\in \{${\em Form 3}$, \mathsf{ISAE}, \mathsf{EPWFP}\}$, then $X$ fails in $\mathcal{V}^{+}_{p}$.
    \item $\mathsf{AC_{n}}$ fails in $\mathcal{V}^{+}_{p}$ for any integer $n\geq 2$. 
    \item $A$ cannot be linearly ordered.
    \item if $X\in \{\mathsf{W_{\aleph_{1}}}, \mathsf{DC_{\aleph_{1}}}\}$, then $X$ fails in $\mathcal{V}^{+}_{p}$.
\end{enumerate}
}
\end{thm}

\begin{proof}
(1). Let $x\in \mathcal{N}_{12}(\aleph_{1})$ with support $E$. So fix$_{\mathcal{G}}(E) \subseteq sym_{\mathcal{G}}(x)$. Then $P = \{\{a\}\}_{a\in E} \cup \{A\backslash E\}$ is a countable partition of $A$, and fix$_{\mathcal{G}}(P)$ = fix$_{\mathcal{G}}(E)$. Thus fix$_{\mathcal{G}}(P) \subseteq sym_{\mathcal{G}}(x)$ and so $x \in \mathcal{V}^{+}_{p}$ with support $P$.

(2). Similarly to the proof of $\mathsf{\neg EPWFP}$ in $\mathcal{V}_{p}$ (cf. the proof of Theorem 4.5(1)), one may verify that if $f$ is a permutation of $A$ in $\mathcal{V}^{+}_{p}$, then the set $\{x \in A : f(x) \neq x\}$ has cardinality at most $\aleph_{0}$. Since $A$ is uncountable, it follows that `for any uncountable $X$, Sym(X) $\neq$ $\aleph_{0}$Sym(X)' fails in $\mathcal{V}^{+}_{p}$. Consequently, if $X\in \{$Form 3$, \mathsf{ISAE}, \mathsf{EPWFP}\}$, then $X$ fails in $\mathcal{V}^{+}_{p}$ by Proposition 4.7(2). 

(3). Fix any integer $n\geq 2$. Similarly to the proof of Theorem 4.5(2), one may verify that the set $[A]^{n}$ has no choice function in $\mathcal{V}^{+}_{p}$. Consequently, $\mathsf{AC_{n}}$ fails in $\mathcal{V}^{+}_{p}$.

(4). Follows from (3).

(5). We can use the arguments in (2) and Proposition 4.7(1) to show that $\mathsf{W_{\aleph_{1}}}$ fails in $\mathcal{V}^{+}_{p}$. The rest follows from the fact that $\mathsf{DC_{\aleph_{1}}}$ implies $\mathsf{W_{\aleph_{1}}}$ in $\mathsf{ZF}$ (cf. \cite[Theorem 8.1(b)]{Jec1973}). However, we write a different argument. In order to show that $\mathsf{W_{\aleph_{1}}}$ fails in $\mathcal{V}^{+}_{p}$, we prove that there is no injection $f$ from $\aleph_{1}$ into $A$. Assume there exists such an $f$ with support $P$, and let $\pi \in $ fix$_{\mathcal{G}}(P)$ be such that $\pi$ moves every atom in each non-singleton block of $P$. Since $P$ contains only countably many singletons, $\pi$ fixes only countably many atoms. Fix $n \in \aleph_{1}$. Since $n$ is in the kernel (the class of all pure sets), we have $\pi (n) = n$. Thus $\pi(f(n)) = f(\pi(n))= f(n)$. But $f$ is one-to-one, and thus, $\pi$ fixes $\aleph_{1}$ many values of $f$ in $A$, a contradiction.
\end{proof}

\begin{remark}
Let $M$ be a model of $\mathsf{ZFA+AC}$ where $A$ is a set of atoms of cardinality $\aleph_{\alpha+1}$. Let $\mathcal{G}$ be the group of all permutations of $A$, $S$ be the set of all $\aleph_{\alpha}$-partitions of $A$, and $\mathcal{F}$ = $\{H:$ $H$ is a subgroup of  $\mathcal{G}$, $H \supseteq$ fix{\em $_{\mathcal{G}}(P)$} for some $P \in S\}$ be a normal filter of subgroups of $\mathcal{G}$. Let $\mathcal{V}_{p}^{\aleph_{\alpha+1}}$ be the permutation model determined by $M$, $\mathcal{G}$, and $\mathcal{F}$. Following the arguments of Theorem 4.8(1)(2), we can see $\mathcal{N}_{12}(\aleph_{\alpha+1}) \subset \mathcal{V}^{\aleph_{\alpha+1}}_{p}$ and $\mathsf{EPWFP}$ fails in $\mathcal{V}^{\aleph_{\alpha+1}}_{p}$.
\end{remark}
%%%%%%%%%%%%%%%

\section{Van Douwen’s Choice Principle in two permutation models}

\begin{prop} 
{\em The following hold:
\begin{enumerate}\item The statement $\mathsf{vDCP\land UT(\aleph_{0}, \aleph_{0}, cuf) \land \neg M(IC, DI)}$ has a permutation model.\item The statement $\mathsf{vDCP\land \neg \mathsf{MC(\aleph_{0}, \aleph_{0})}}$ has a permutation model.
\end{enumerate}}\end{prop}

\begin{proof}(1) We recall the permutation model $\mathcal{N}$ which was constructed in \cite[proof of Theorem 3.3]{CHHKR2008} where $\mathsf{UT(\aleph_{0}, \aleph_{0}, cuf)}$ holds. In order to describe $\mathcal{N}$, we start with a model $M$ of $\mathsf{ZFA + AC}$ with a set $A$ of atoms such that $A$ has a denumerable partition $\{A_{i} : i \in \omega\}$ into denumerable sets, and for each $i \in \omega$, $A_{i}$ has a denumerable partition $P_{i} = \{A_{i,j} : j \in \mathbb{N}\}$ into finite sets such that, for every $j \in \mathbb{N}$, $\vert A_{i,j}\vert = j$. Let $\mathcal{G} = \{\phi \in Sym(A) : (\forall i \in \omega)(\phi(A_{i}) = A_{i})$ and $\vert\{x \in A : \phi(x) \neq x\}\vert < \aleph_{0}\}$, where $Sym(A)$ is the group of all permutations of $A$. Let $\textbf{P}_{i} = \{\phi(P_{i}) : \phi \in \mathcal{G}\}$ for each $i \in \omega$ and let $\textbf{P} = \bigcup\{\textbf{P}_{i} : i \in \omega\}$. Let $\mathcal{F}$ be the normal filter of subgroups of $\mathcal{G}$ generated by the filter base $\{$fix$_{\mathcal{G}}(E) : E \in [\textbf{P}]^{<\omega}\}$. Then $\mathcal{N}$ is the permutation model determined by $M$, $\mathcal{G}$ and $\mathcal{F}$. Keremedis, Tachtsis, and Wajch proved that $\mathsf{M(IC, DI)}$ fails in $\mathcal{N}$ (cf. \cite[proof of Theorem 13(i)]{KTW2021}). We follow steps (1), (2), and (4) from the proof of \cite[Lemma 5.1]{Ban2} to see that $\mathsf{vDCP}$ holds in $\mathcal{N}$. For the sake of convenience, we write down the proof.

\begin{lem}
{\em If $(X,\leq)$ is a poset in $\mathcal{N}$, then $X$ can be written as a well-ordered disjoint union $\bigcup\{W_{\alpha} : \alpha<\kappa\}$ of antichains.}
\end{lem}

\begin{proof}
Let $(X,\leq)$ be a poset in $\mathcal{N}$ and $E\in [\textbf{P}]^{<\omega}$ be a support of $(X,\leq)$. 
Following the arguments of \cite[claim 3.5]{Tac2019b} we can see that for each $p \in X$, the set $\text{Orb}_{\text{fix}_{\mathcal{G}}(E)}(p)$ is an antichain in $(X,\leq)$ since every element $\eta\in \mathcal{G}$ moves only finitely many atoms. Following the arguments of \cite[claim 3.6]{Tac2019b} we can see that

\begin{center}
    $\mathcal{O}=\{\text{Orb}_{\text{fix}_{\mathcal{G}}(E)}(p): p\in X\}$
\end{center}

is a well-ordered partition of $X$.
\end{proof}

We recall the arguments from the $1^{st}$-paragraph of \cite[p.175]{HST2016} to give a proof of $\mathsf{vDCP}$ in $\mathcal{N}$.
Let $\mathcal{A} = \{(A_{i}, \leq_{i}) : i \in I\}$ be a family as in $\mathsf{vDCP}$. Without loss of generality, we assume that $\mathcal{A}$ is pairwise disjoint. Let $R = \bigcup \mathcal{A}$. We partially order $R$ by requiring $x \prec y$ if and only if there exists an index $i\in I$ such that
$x, y \in A_{i}$ and $x \leq_{i} y$. By Lemma 5.2, $R$ can be written as a well-ordered disjoint
union $\bigcup\{W_{\alpha} : \alpha<\kappa\}$ of antichains. For each $i \in I$, let $\alpha_{i} = min\{\alpha \in \kappa : A_{i} \cap W_{\alpha} \neq \emptyset\}$. Since for all $i\in I$, $A_{i}$ is linearly ordered, it follows that $A_{i} \cap W_{\alpha_{i}}$ is a singleton for each $i\in I$. Consequently, $f = \{(i,\bigcup(A_{i} \cap W_{\alpha_{i}})) : i \in I\}$ is a choice function of $\mathcal{A}$. Thus, $\mathsf{vDCP}$ holds in $\mathcal{N}$.

(2). We recall the permutation model (say $\mathcal{M}$) which was constructed in \cite[proof of Theorem 3.4]{HT2021}. In  order  to  describe $\mathcal{M}$, we
start with a model $M$ of $\mathsf{ZFA + AC}$ with a denumerable set $A$ of atoms which is written as a disjoint union
$\bigcup\{A_{n} : n \in \omega\}$, where $\vert A_{n}\vert = \aleph_{0}$ for all $n \in \omega$.
For each $n \in \omega$, we let $\mathcal{G}_{n}$ be the group of all permutations of $A_{n}$ which move only finitely
many elements of $A_{n}$. Let $\mathcal{G}$ be the weak direct product of the $\mathcal{G}_{n}$’s for $n \in \omega$.
Consequently, every permutation of $A$ in $\mathcal{G}$ moves only finitely many atoms.
Let $\mathcal{I}$ be the normal ideal of subsets of $A$ which is generated by finite unions of $A_{n}$’s. Let $\mathcal{F}$ be the normal filter on $\mathcal{G}$ generated by the subgroups fix$_{\mathcal{G}}(E)$, $E \in \mathcal{I}$.
Let $\mathcal{M}$ be the Fraenkel–Mostowski model, which is determined by $M$, $\mathcal{G}$, and $\mathcal{F}$. 

Howard and Tachtsis \cite[proof of Theorem 3.4]{HT2021} proved that $\mathsf{MC(\aleph_{0}, \aleph_{0})}$ fails in $\mathcal{M}$. Since every permutation of $A$ in $\mathcal{G}$ moves only finitely many atoms, following the arguments in the proof of (1), $\mathsf{vDCP}$ holds in $\mathcal{M}$.
\end{proof}

\begin{remark}
In every Fraenkel-Mostowski permutation model, $\mathsf{CS}$ (Every poset without a maximal element has two disjoint cofinal subsets) implies $\mathsf{vDCP}$ (cf. \cite[Theorem 3.15(3)]{HST2016}). We can see that in the above-mentioned permutation models (i.e., $\mathcal{N}$ and $\mathcal{M}$) $\mathsf{CS}$ and $\mathsf{CWF}$ (Every poset has a cofinal well-founded subset) hold applying Lemma 5.2 and following the methods of  \cite[Theorem 3.26]{HST2016} and \cite[proof of Theorem 10 (ii)]{Tac2018}.
\end{remark}

%%%%%%%%%%%%%%%%%
\section{Spanning subgraphs and weak choice forms} 
We recall some known results.
\begin{lem}{($\mathsf{ZF}$; Delhomme--Morillon; \cite[Lemma 1]{DM2006})}
{\em Given a set $X$ and a set $A$ which is the range of no mapping with domain $X$, consider a mapping $f : A \rightarrow \mathcal{P}(X)\backslash \{\emptyset\}$. Then
\begin{enumerate}
    \item  There are distinct $a$ and $b$ in $A$ such that $f(a)\cap f(b) \neq \emptyset$.
    \item If the set $A$ is infinite and well-orderable, then for every positive integer $p$, there is an $F \in [A]^{p}$ such that $\bigcap f[F]:=\bigcap_{a\in F} f(a)$
is non-empty.
\end{enumerate}
}
\end{lem}

\begin{lem}(Cayley’s formula; $\mathsf{ZF}$)
{\em The number of spanning trees in $K_{n}$ is $n^{n-2}$ for any $n\in \omega\backslash \{0,1,2\}$.}
\end{lem}

\begin{lem}(Scoin's formula; $\mathsf{ZF}$)
{\em The number of spanning trees in $K_{m,n}$ is $n^{m-1}m^{n-1}$ for any $n,m\in \omega\backslash \{0,1\}$.}
\end{lem}

\begin{lem}(cf. \cite[Chapter 30, Problem 5]{KT2006})
{\em $\mathsf{AC}_{m}$ implies $\mathsf{AC}_{n}$ if $m$ is a multiple of $n$.}
\end{lem}

Applying the above lemmas we prove the following propositions.
\begin{prop} ($\mathsf{ZF}$) {\em The following hold:
\begin{enumerate}
    \item $\mathsf{AC_{\leq n-1}^{\omega}} + \mathcal{Q}^{n}_{lf,c}$
    is equivalent to $\mathsf{AC_{fin}^{\omega}}$ for any $2< n\in \omega$.

    \item $\mathsf{UT(WO,WO,WO)}$ implies $\mathsf{AC_{\leq n-1}^{WO}} + \mathcal{Q}_{lw,c}^{n,k}$ and the later implies $\mathsf{AC_{WO}^{WO}}$ for any $2< n,k\in \omega$. 
    \item $\mathcal{P}_{lf,c}^{m}$ is equivalent to $\mathsf{AC_{fin}^{\omega}}$ for any even integer $m\geq 4$.
    \item $\mathcal{Q}^{n}_{lf,c}$ fails in $\mathcal{N}_{6}$ for any $2<n\in\omega$.
\end{enumerate}}
\end{prop}

\begin{proof}
(1). ($\Leftarrow$) We assume  $\mathsf{AC_{fin}^{\omega}}$. Fix any $2< n\in \omega$. We know that $\mathsf{AC_{fin}^{\omega}}$ implies $\mathsf{AC_{\leq n-1}^{\omega}}$  in $\mathsf{ZF}$. Moreover, $\mathsf{AC_{fin}^{\omega}}$ implies $\mathcal{Q}^{n}_{lf,c}$ in $\mathsf{ZF}$ (cf. \cite[Theorem 2]{DM2006}).

($\Rightarrow$) Fix any $2< n\in \omega$. We show that $\mathsf{AC_{\leq n-1}^{\omega}} + \mathcal{Q}^{n}_{lf,c}$ implies $\mathsf{AC_{fin}^{\omega}}$ in $\mathsf{ZF}$. 
Let $\mathcal{A}=\{A_{i}:i\in \omega\}$ be a countably infinite set of non-empty finite sets. Without loss of generality, we assume that $\mathcal{A}$ is disjoint. Let $A=\bigcup_{i\in \omega} A_{i}$. Consider a countably infinite family $(B_{i},<_{i})_{i\in \omega}$ of well-ordered sets such that $\vert B_{i}\vert=\vert A_{i}\vert + k$ for some fixed $1\leq k\in \omega$, for each $i\in \omega$, $B_{i}$ is disjoint from $A$ and the other $B_{j}$’s, and there is no mapping with domain $A_{i}$ and range $B_{i}$ (cf. the proof of \cite[Theorem 1, Remark 6]{DM2006}).
Let $B=\bigcup_{i\in \omega} B_{i}$. Consider another countably infinite sequence $T=\{t_{i}:i\in \omega\}$ disjoint from $A$ and $B$. We construct a graph $G_{1}=(V_{G_{1}}, E_{G_{1}})$ (cf. Figure 1).

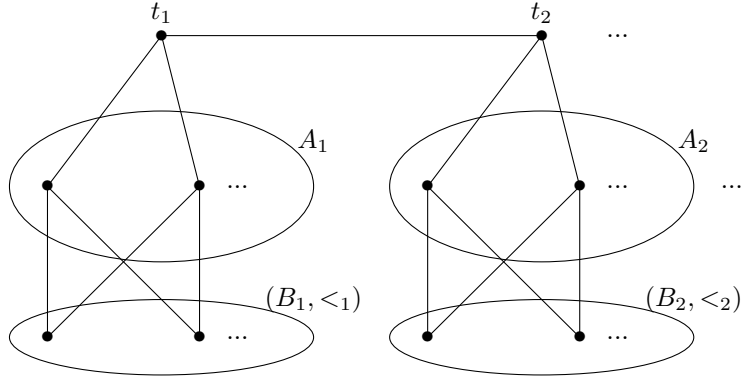
\begin{figure}[!ht]
\centering
\begin{minipage}{\textwidth}
\centering
\begin{tikzpicture}
\draw (-2.5, 1) ellipse (2 and 1);
\draw (-4,1) node {$\bullet$};
\draw (-2,1) node {$\bullet$};
\draw (-1.5,1) node {...};
\draw (-0.5,1.6) node {$A_{1}$};
\draw (-2.5,3) node {$\bullet$};%t_{1} bullet
\draw (-2.5,3.3) node {$t_{1}$};
%%%%
\draw (-2.5, -1) ellipse (2 and 0.5);
\draw (-0.5,-0.5) node {$(B_{1},<_{1})$};
\draw (-4,-1) node {$\bullet$};
\draw (-2,-1) node {$\bullet$};
\draw (-1.5,-1) node {...};
%%%%%%%%%%%%%%%%%%%
\draw (-4,1) -- (-2.5,3);%a_{1,1} to r_{1}
\draw (-2,1) -- (-2.5,3);
%%%%%%%
\draw (-4,1)-- (-4,-1);%a_{1,1} to b_{1,1}
\draw (-2,1)-- (-2,-1);
\draw (-4,1)-- (-2,-1);%a_{1,1} to b_{1,1}
\draw (-2,1)-- (-4,-1);
%%%%%%%
\draw (2.5, 1) ellipse (2 and 1);
\draw (1,1) node {$\bullet$};
\draw (3,1) node {$\bullet$};
%%%%%%%%%
\draw (3.5,1) node {...};
\draw (4.5,1.6) node {$A_{2}$};
\draw (2.5,3) node {$\bullet$};%t_{2} bullet
\draw (3.5,3) node {...};
\draw (2.5,3.3) node {$t_{2}$};
%%%%%%
\draw (2.5, -1) ellipse (2 and 0.5);
\draw (4.5,-0.5) node {$(B_{2},<_{2})$};
\draw (1,-1) node {$\bullet$};
\draw (3,-1) node {$\bullet$};
\draw (3.5,-1) node {...};
%%%%%%%%%%%%%%%%%%%
\draw (3,1) -- (2.5,3);%t_{2} to a_{2,1}
\draw (1,1) -- (2.5,3);
%%%%%%%
\draw (3,1)-- (3,-1);%a_{2,1} to b_{2,1,1}
\draw (1,1)-- (1,-1);
\draw (3,1)-- (1,-1);%a_{2,1} to b_{2,1,1}
\draw (1,1)-- (3,-1);
%%%%%%%%%%%%%%
\draw (-2.5,3) -- (2.5,3);
\draw (5,1) node {...};
\end{tikzpicture}
\end{minipage}
\caption{\em The graph $G_{1}$}
\end{figure}

\begin{center}
Let $V_{G_{1}} = A\cup B\cup T$, 
and

$E_{G_{1}} = \{\{t_{i}, t_{i+1}\}:{i\in\omega}\} 
\text{ }\bigcup  \text{ }
\{\{t_{i}, x\}: i\in\omega,x\in A_{i}\}
\text{ }\bigcup \text{ }
\{\{x,y\}:i\in \omega, x\in A_{i},y\in B_{i}\}$.\footnote{i.e., For each $i \in\omega$, let $\{t_{i},t_{i+1}\}\in E_{G_{1}}$ and $\{t_{i},x\}\in E_{G_{1}}$ for every element $x\in A_{i}$. Also for each $i \in\omega$, join each $x\in A_{i}$ to every element of $B_{i}$.}
\end{center}

Clearly, the graph $G_{1}=(V_{G_{1}}, E_{G_{1}})$ is connected and locally finite. By assumption, $G_{1}$ has a spanning subgraph $G'_{1}$ omitting $K_{2,n}$.
For each $i \in \omega$, let $f_{i}: B_{i} \rightarrow \mathcal{P}(A_{i})\backslash\{\emptyset\}$ map each element of $B_{i}$ to its neighbourhood in $G'_{1}$. We can see that for any two distinct $\epsilon_{1}$ and $\epsilon_{2}$ in $B_{i}$, $f_{i}(\epsilon_{1})\cap f_{i}(\epsilon_{2})$ has at most $n-1$
elements, since $G'_{1}$ has no $K_{2,n}$. By Lemma 6.1(1), there are tuples $(\epsilon'_{1}, \epsilon'_{2})\in B_{i}\times B_{i}$ such that $f_{i}(\epsilon'_{1})\cap f_{i}(\epsilon'_{2})\neq \emptyset$. Consider the first such tuple $(\epsilon''_{1},\epsilon''_{2})$ with respect to the well-ordering on $B_{i}\times B_{i}$.
Let $A'_{i}=f_{i}(\epsilon''_{1})\cap f_{i}(\epsilon''_{2})$.
By $\mathsf{AC_{\leq n-1}^{\omega}}$, we can obtain a choice function of $\mathcal{A}'=\{A'_{i}:i\in \omega\}$, which is a choice function of $\mathcal{A}$.

(2). For the first implication, we know that $\mathsf{UT(WO,WO,WO)}$ implies $\mathsf{AC_{\leq n-1}^{WO}}$ as well as the statement {\em `Every locally well-orderable connected graph is well-orderable'} in $\mathsf{ZF}$. The rest follows from the fact that every well-ordered connected graph has a spanning tree in $\mathsf{ZF}$.

We show that $\mathsf{AC_{\leq n-1}^{WO}}$ + $\mathcal{Q}_{lw,c}^{n,k}$ implies $\mathsf{AC_{WO}^{WO}}$. Let $\mathcal{A}=\{A_{n}:n\in \kappa\}$ be a well-orderable set of non-empty well-orderable sets. Without loss of generality, we assume that $\mathcal{A}$ is disjoint. Let $A=\bigcup_{i\in \kappa} A_{i}$. Consider an infinite well-orderable family $(B_{i},<_{i})_{i\in \kappa}$ of infinite well-orderable sets such that for each $i\in \kappa$, $B_{i}$ is disjoint from $A$ and the other $B_{j}$’s, and there is no mapping with domain $A_{i}$ and range $B_{i}$ (cf. the proof of \cite[Theorem 1, Remark 6]{DM2006}). Let $B=\bigcup_{i\in \kappa} B_{i}$. Consider another $\kappa$-sequence $T=\{t_{n}:n\in \kappa\}$ disjoint from $A$ and $B$. We construct a graph $G_{2}=(V_{G_{2}}, E_{G_{2}})$ as follows:

\begin{center}
Let $V_{G_{2}} = A\cup B\cup T$, 
and

$E_{G_{2}} = \{\{t_{i}, t_{j}\}: i,j\in\kappa$ and $i\neq j\}
\text{ }\bigcup  \text{ }
\{\{t_{i}, x\}: i\in\kappa,x\in A_{i}\}
\text{ }\bigcup \text{ }
\{\{x,y\}:i\in \kappa, x\in A_{i},y\in B_{i}\}$.
\end{center}

Clearly, the graph $G_{2}$ is connected and locally well-orderable. By assumption, $G_{2}$ has a spanning subgraph $G'_{2}$ omitting $K_{k,n}$. For each $i \in \kappa$, let $f_{i}: B_{i} \rightarrow \mathcal{P}(A_{i})\backslash\{\emptyset\}$ map each element of $B_{i}$ to its neighbourhood in $G'_{2}$. We can see that for any finite $k$-subset $H_{i}\subseteq B_{i}$, $\bigcap_{\epsilon\in H_{i}} f_{i}(\epsilon)$ has at most $n-1$ elements, since $G'_{2}$ has no $K_{k,n}$. Since each $B_{i}$ is infinite and well-orderable, by Lemma 6.1(2), there are tuples $(\epsilon_{1}, \epsilon_{2},...\epsilon_{k})\in B_{i}^{k}$ such that $\bigcap_{1\leq i< k}f_{i}(\epsilon_{i})\neq \emptyset$. Consider the first such tuple $(\epsilon'_{1}, \epsilon'_{2},...\epsilon'_{k})$ with respect to the well-ordering on $B_{i}^{k}$. Let $A'_{i}=\bigcap_{1\leq i< k}f_{i}(\epsilon'_{i})$. By $\mathsf{AC_{\leq n-1}^{WO}}$, we can obtain a choice function of $\mathcal{A}'=\{A'_{n}:n\in \kappa\}$, which is a choice function of $\mathcal{A}$.

(3). 
($\Rightarrow$) Fix any even integer $m=2(k+1) \geq 4$. We prove that $\mathcal{P}_{lf,c}^{m}$ implies $\mathsf{AC_{fin}^{\omega}}$. 
Let $\mathcal{A}=\{A_{i}:i\in\omega\}$ be a countably infinite set of non-empty finite sets and $A=\bigcup_{i\in \omega} A_{i}$. Let $T=\{t_{i}:i\in\omega\}$ be a sequence such that $t_{i}$’s are pair-wise distinct and belong to no $A_{j} \times \{1,..., k\}$, and $R=\{r_{i}:i\in\omega\}$ be a sequence such that $r_{i}$'s are pair-wise distinct and belong to no $(A_{j} \times \{1,..., k\})\cup \{t_{j}\}$ for any $i,j\in\omega$. We construct a graph $G_{3}= (V_{G_{3}},E_{G_{3}})$ as follows  (cf. Figure 2):
\begin{center}
    \text{Let } $V_{G_{3}} = (\bigcup_{i\in\omega} (A_{i}\times \{1,...,k\})) \cup R \cup T$, \text{and }
    
    $E_{G_{3}} = (\bigcup_{i\in \omega,x\in A_{i}} \{\{r_{i},(x, 1)\}, \{(x, 1),(x, 2)\},..., \{(x, k),t_{i}\}\}) \text{ }\bigcup\text{ } \{\{r_{i},r_{i+1}\}:i\in \omega\}$.
\end{center}

\begin{figure}[!ht]
\centering
\begin{minipage}{\textwidth}
\centering
\begin{tikzpicture}[scale=1]
%\draw (-2.5, 1) ellipse (2 and 1);
\draw (-4,1) node {$\bullet$};
\draw (-3,1) node {$\bullet$};
\draw (-2,1) node {$\bullet$};
\draw (-1.5,1) node {...};
%\draw (-0.5,1.6) node {$A_{1}$};
%\draw (-2.5,-1) node {$\bullet$};
\draw (-2.5,3) node {$\bullet$};%r_{1} bullet
\draw (-2.5,3.3) node {$t_{1}$};
%%%
\draw (-4,-0.4) node {$\bullet$};
\draw (-3,-0.4) node {$\bullet$};
\draw (-2,-0.4) node {$\bullet$};
%%%%
\draw (-4,-1) node {$\bullet$};
\draw (-3,-1) node {$\bullet$};
\draw (-2,-1) node {$\bullet$};
%%%%%
\draw (-4,-1.9) node {$\bullet$};%b_{1,p,1} bullet
\draw (-3,-1.9) node {$\bullet$};
\draw (-2,-1.9) node {$\bullet$};
%%%
\draw (-2.5,-3.1) node {$r_{1}$};
\draw (-2.5,-2.8) node {$\bullet$};%t_{1} bullet
%%%%%%%%%%%%%%%%%%%
\draw (-1.5,-2.1) node {$(x_{1},1)$};
\draw (-1.5,-1.3) node {$(x_{1},2)$};
\draw (-1.5,0.8) node {$(x_{1},k)$};
\draw (-5.5,0) node {$A_{1}\times \{1,...,k\}$};
%%%%%%%%%%%%%%%
\draw (-1.5,-0.4) node {...};
\draw (-1.5,-1) node {...};
\draw (-1.5,-1.9) node {...};
%%%%%%%%%%%%%%%%%%%
\draw (-4.2,-2.5) -- (-4.2,1.2) -- (0,1.2) -- (0,-2.5) -- (-4.2,-2.5);
%%%%%%%%
\draw (-4,1) -- (-2.5,3);%a_{1,1} to r_{1}
\draw (-3,1) -- (-2.5,3);
\draw [ultra thick](-2,1) -- (-2.5,3);
%%%%%%%
\draw (-4,1)-- (-4,-0.4);%a_{1,1} to b_{1,1,1}
\draw (-3,1)-- (-3,-0.4);
\draw [ultra thick](-2,1)-- (-2,-0.4);
%%%%%%%
\draw (-4,-0.4)-- (-4,-1);%b_{1,1,1} to b_{1,2,1}
\draw (-3,-0.4)-- (-3,-1);
\draw [ultra thick](-2,-0.4)-- (-2,-1);
%%%
\draw (-4,-1.3) node {.};
\draw (-4,-1.4) node {.};
\draw (-4,-1.5) node {.};
\draw (-3,-1.3) node {.};
\draw (-3,-1.4) node {.};
\draw (-3,-1.5) node {.};
\draw (-2,-1.3) node {.};
\draw (-2,-1.4) node {.};
\draw (-2,-1.5) node {.};
%%%
\draw (-4,-1.9) -- (-2.5,-2.8);%b_{1,p,1} to t_{1}
\draw (-3,-1.9) -- (-2.5,-2.8);
\draw [ultra thick](-2,-1.9) -- (-2.5,-2.8);
%%%%%
%\draw (2.5, 1) ellipse (2 and 1);
\draw (1,1) node {$\bullet$};
\draw (2,1) node {$\bullet$};
\draw (3,1) node {$\bullet$};

\draw (-3,1) node {$\bullet$};
\draw (-2,1) node {$\bullet$};

\draw (3.5,1) node {...};
%\draw (4.5,1.6) node {$A_{2}$};
%\draw (2.5,-1) node {$\bullet$};
%\draw (2.5,-1.3) node {$t_{2}$};
\draw (2.5,3) node {$\bullet$};%r_{2} bullet
\draw (2.5,3.3) node {$t_{2}$};
%%%%%%
\draw (1,-0.4) node {$\bullet$};%b_{2,1,1} bullet
\draw (2,-0.4) node {$\bullet$};
\draw (3,-0.4) node {$\bullet$};
%%%%
\draw (1,-1) node {$\bullet$};
\draw (2,-1) node {$\bullet$};
\draw (3,-1) node {$\bullet$};
%%%%%
\draw (1,-1.9) node {$\bullet$};%b_{2,p,1} bullet
\draw (2,-1.9) node {$\bullet$};
\draw (3,-1.9) node {$\bullet$};
%%%
\draw (2.5,-3.1) node {$r_{2}$};
\draw (2.5,-2.8) node {$\bullet$};%t_{2} bullet
%%%%%%%%%%%%%%%%%%%
\draw (3,1) -- (2.5,3);%r_{2} to a_{2,1}
\draw (2,1) -- (2.5,3);
\draw (1,1) -- (2.5,3);
%%%%%%%
\draw (3,1)-- (3,-0.4);%a_{2,1} to b_{2,1,1}
\draw (2,1)-- (2,-0.4);
\draw (1,1)-- (1,-0.4);
%%%%%%%
\draw (3,-0.4)-- (3,-1);%b_{2,1,1} to b_{2,2,1}
\draw (2,-0.4)-- (2,-1);
\draw (1,-0.4)-- (1,-1);
%%%
\draw (1,-1.3) node {.};
\draw (1,-1.4) node {.};
\draw (1,-1.5) node {.};
\draw (2,-1.3) node {.};
\draw (2,-1.4) node {.};
\draw (2,-1.5) node {.};
\draw (3,-1.3) node {.};
\draw (3,-1.4) node {.};
\draw (3,-1.5) node {.};
%%%
\draw (1,-1.9) -- (2.5,-2.8);%b_{1,p,1} to t_{2}
\draw (2,-1.9) -- (2.5,-2.8);
\draw (3,-1.9) -- (2.5,-2.8);
%%%%%%%%%%%%%%
\draw (-2.5,-2.8) -- (2.5,-2.8);
\draw (3.5,-2.8) node {...};
\draw (5,1) node {...};
\end{tikzpicture}
\end{minipage}
\caption{\em The graph $G_{3}$, where $\{t_{1}, (x_{1},k),... (x_{1}, 1), r_{1}\}$ is a path in $\zeta$.}
\end{figure}
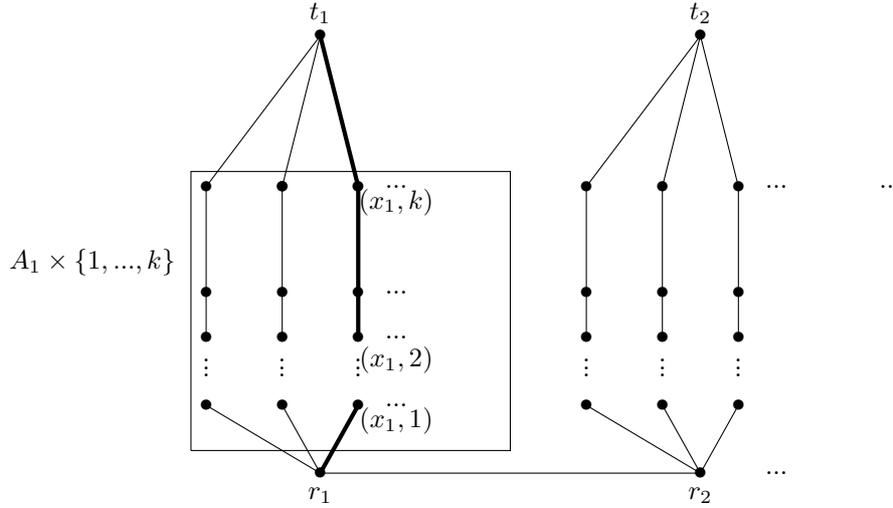
%{\bf{\underline{Constructing $G_{3}$}:}} 

Clearly, $G_{3}$ is locally finite and connected. By assumption, $G_{3}$ has a spanning $m$-bush $\zeta$. We can see that $\zeta$ generates a choice function of $\mathcal{A}$: for each $i \in I$, there is a unique $x \in A_{i}$, say $x_{i}$, such that $\{t_{i}, (x_{i},k),... (x_{i}, 1), r_{i}\}$ is a path in $\zeta$.

($\Leftarrow$) Fix any even integer $m\geq 4$. We prove that $\mathsf{AC_{fin}^{\omega}}$ implies  $\mathcal{P}_{lf,c}^{m}$. We know that $\mathsf{AC_{fin}^{\omega}}$ implies the statement {\em `Every infinite locally finite connected graph is countably infinite'} in $\mathsf{ZF}$. The rest follows from the fact that every well-ordered graph has a spanning tree in $\mathsf{ZF}$ and any spanning tree is a spanning $m$-bush.

(4). In $\mathcal{N}_{6}$, $\mathsf{AC_{fin}^{\omega}}$ fails, whereas $\mathsf{AC_{\leq n-1}^{\omega}}$ holds for any natural number $n> 2$ (cf. \cite[proof of Theorem 7.11, p.110]{Jec1973}). By Proposition 6.5(1), $\mathcal{Q}^{n}_{lf,c}$ fails in the model.
\end{proof}

We recall the definition of $P'_{G}$ from $\S 1.4.1$. 

\begin{prop}($\mathsf{ZF}$) 
{\em Fix any $2< k\in \omega$ and any $2\leq p,q< \omega$. The following hold:
\begin{enumerate}
\item $\mathsf{AC_{k^{k-2}}}$ implies the statement `Every graph from the class $P'_{K_{k}}$ has a spanning tree'.

\item $\mathsf{AC_{k}}$ implies the statement `Every graph from the class $P'_{C_{k}}$ has a spanning tree'.

\item ($\mathsf{AC_{p^{q-1}q^{p-1}}+AC_{p+q}}$) implies the statement `Every graph from the class $P'_{K_{p,q}}$ has a spanning tree'.
\end{enumerate}
}
\end{prop}

\begin{proof}
We prove (1). Let $G_{2}=(V_{G_{2}}, E_{G_{2}})$ be a graph from the class $P'_{K_{k}}$. Then there is a $G_{1}\in P_{K_{k}}$ (an infinite graph whose only components are $K_{k}$) such that $V_{G_{2}}=V_{G_{1}}\cup \{t\}$ for some $t\not\in V_{G_{1}}$. Let $\{A_{i}: i\in I\}$ be the components of $G_{1}$. By $\mathsf{AC_{k}}$ (which follows from $\mathsf{AC_{k^{k-2}}}$ (cf. Lemma 6.4)), we choose a sequence of vertices $\{a_{i}:i\in I\}$ such that $a_{i}\in A_{i}$ for all $i\in I$. By Lemma 6.2, the number of spanning trees in $A_{i}$ is $k^{k-2}$ for any $i\in I$. By $\mathsf{AC_{k^{k-2}}}$, we choose a sequence $\{s_{i}:i\in I\}$ such that $s_{i}=(V_{s_{i}}, E_{s_{i}})$ is a spanning tree of $A_{i}$ for all $i\in I$. We construct a graph $S_{G_{2}}=(V_{S_{G_{2}}}, E_{S_{G_{2}}})$ as follows:

\begin{center}
     \text{Let } $V_{S_{G_{2}}}=\{t\}\cup \bigcup_{i\in I} V_{s_{i}}$, \text{ and } 
     
     $E_{S_{G_{2}}}=\{\{t, a_{i}\}: i\in I\}\cup \bigcup_{i\in I} E_{s_{i}}$.
\end{center}

Then the graph $S_{G_{2}}=(V_{S_{G_{2}}}, E_{S_{G_{2}}})$ is a spanning tree of $G_{2}$.
Similarly, we can prove (2) and (3) since the number of spanning trees in $C_{k}$ is $k$ and the number of spanning trees in $K_{p,q}$ is $p^{q-1}q^{p-1}$ (cf. Lemma 6.3).
\end{proof}
%%%%%%%%%%%%%%%%%%%
\section{Synopsis of results, questions, and further studies} 
%%%%%%%%%%
%%%%%%%%%%%%%%%%%%%
\subsection{Synopsis of results} Fix any $2< n,k\in\omega$, any $2\leq q\in\omega$, and any even integer $m\geq 4$. In Figure 3, known results are depicted with dashed arrows, new results in $\mathsf{ZF}$ are mentioned with simple arrows, and new results in $\mathsf{ZFA}$ are mentioned with thick dotted arrows.

\begin{figure}[!ht]
\begin{minipage}{\textwidth}
\begin{tikzpicture}[scale=6]
\draw (2.42,1.02) node[above] {$\mathsf{AC}$};
\draw[dashed, -triangle 60] (2.35,1.05) -- (2.22,1.05);
\draw (2.15,1) node[above] {$\mathsf{AC_{q}}$};
\draw[dashed, -triangle 60] (2.1,1.1) -- (1.85,1.2);
\draw[dashed, -triangle 60] (2.15,1.1) -- (2.45, 1.2);
\draw (1.85,1.2) node[above] {$\mathsf{AC_{q}^{-}}$};
\draw (2.45, 1.2) node[above] {``There are no amorphous sets"};
\draw (1.85,1.5) node[above] {\textbf{Form 304}};
\draw (2.45,1.5) node[above] {\textbf{Form 233}};
\draw[ultra thick, dotted, -triangle 60] (1.85,1.5) -- (1.85, 1.3);
\draw[dashed, -triangle 60] (2.45,1.5) -- (2.45, 1.3);
\draw[ultra thick, dotted, -triangle 60] (2.45,1.5) -- (1.9, 1.3);
\draw[dashed, -triangle 60] (1.85,1.5) -- (2.4, 1.3);
\draw (1.87,1.42) -- (1.83,1.48);
\draw (2.37,1.42) -- (2.33,1.48);
\draw (2.47,1.48) -- (2.43,1.42);
\draw (2.03,1.48) -- (1.99,1.42);
%%%%%%%%%%%%%%%%%%%%%%%%%%%%
\draw[dashed, -triangle 60] (2.5,1.05) -- (2.65,1.05);
\draw (2.94,1) node[above] {$\mathsf{UT(WO,WO,WO)}$};
%%%%%%%%
\draw[-triangle 60] (3.22,1.05) -- (3.4,1.05);%UT(WO,WO,WO) to ...
\draw (3.65,1) node[above] {$\mathsf{AC_{\leq n-1}^{WO}}$ + $\mathcal{Q}_{lw,c}^{n,k}$};
\draw[-triangle 60] (3.88,1.06) -- (4.1,1.06);% ... to AC_{WO}^{WO}
%%%%%
\draw (4.2,1.2) node[above] {$\mathsf{AC_{fin}^{\omega}}$};
%%%%%%%%%%%
\draw (4.2, 1) node[above] { $\mathsf{AC^{WO}_{WO}}$};
\draw[dashed,-triangle 60] (4.2,1.1) --(4.2,1.22);%AC_{\WO}WO to AC_{\aleph_{fin}^{\omega}
%%%%%%%%%%
\draw[-triangle 60] (4.2,1.3)--(4.2,1.48);
\draw[-triangle 60] (4.2,1.48)--(4.2,1.3);
\draw (4.1,1.48) node[above] {$\mathsf{AC_{\leq n-1}^{\omega}}$ + $\mathcal{Q}_{lf,c}^{n}$};
%%%%%%
\draw[-triangle 60] (3.6,1.48)--(4.1, 1.3);
\draw[-triangle 60] (4.1, 1.3)--(3.6,1.48);
\draw (3.6,1.48) node[above] {$\mathcal{P}_{lf,c}^{m}$};
%%%%%%
\end{tikzpicture}
\end{minipage}
\caption{\em In the above figure,  we summarize the results of this note from sections 3, 6.}
\end{figure}
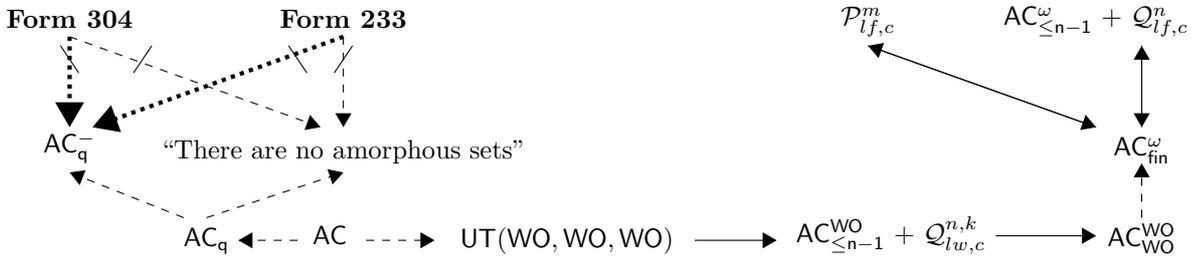

%%%%%%%%%
\textbf{Consistency results:}

\begin{itemize}
\item $\mathcal{N}_{HT}^{1}(q)\models\neg$ Form 269 $\land$ Form 304 $\land$ Form 233. Consequently, for any integer $q\geq 2$, neither Form 233 nor Form 304 implies $\mathsf{AC_{q}^{-}}$ in $\mathsf{ZFA}$.

\item $\mathsf{AC^{LO}}\nrightarrow$ Form 269 in $\mathsf{ZFA}$.

\item $\mathcal{V}_{p}\models \neg$ Form 3 $\land \neg\mathsf{ISAE}
\land\neg\mathsf{EPWFP} \land\neg\mathsf{AC_{q}} \land\neg\mathsf{MA(\aleph_{0})} \land\neg\mathsf{\leq\aleph_{0}}$-$\mathsf{MC}$.

\item $\mathcal{V}^{+}_{p}\models \neg$ Form 3 $\land \neg\mathsf{ISAE}
\land\neg\mathsf{EPWFP} \land\neg\mathsf{AC_{q}} \land\neg\mathsf{W_{\aleph_{1}}} \land\neg\mathsf{DC_{\aleph_{1}}}$.
\item The statements `$\mathsf{vDCP\land UT(\aleph_{0}, \aleph_{0}, cuf) \land \neg M(IC, DI)}$' and `$\mathsf{vDCP\land \neg \mathsf{MC(\aleph_{0}, \aleph_{0})}}$' have permutation models.
\end{itemize}
%%%%%%%%%%%%%
\subsection{Questions, and further studies}
\begin{question}
Which other choice principles hold in $\mathcal{V}_{p}$? In particular does $\mathsf{CAC}$, the infinite Ramsey's Theorem ($\mathsf{RT}$) \cite[Form 17]{HR1998},  and Form 233 hold in $\mathcal{V}_{p}$?
\end{question}

Bruce \cite{Bru2016} proved that $\mathsf{UT(WO,WO,WO)}$ holds in $\mathcal{V}_{p}$.

\begin{question}
Does $\mathsf{UT(WO,WO,WO)}$ hold in $\mathcal{V}^{+}_{p}$?
\end{question}

We recall that every symmetric extension (symmetric submodel of a forcing extension where $\mathsf{AC}$ can consistently fail) is given by a symmetric system $\langle \mathbb{P}, \mathcal{G}, \mathcal{F}\rangle$, where $\mathbb{P}$ is a forcing notion, $\mathcal{G}$ is a group of permutations of $\mathbb{P}$, and $\mathcal{F}$ is a normal filter of subgroups over $\mathcal{G}$. We recall the definition of Feferman--L\'{e}vy's symmetric extension from Dimitriou's Ph.D. thesis (cf. \cite[Chapter 1, section 2]{Dim2011}).

\textbf{Forcing notion $\mathbb{P}_{1}$:} Let $\mathbb{P}_{1} = \{p: \omega \times \omega \rightharpoonup \aleph_{\omega}: (\vert p\vert < \omega$ and $\forall (n, i) \in dom(p), p(n, i) < \omega_{n})\}$ be a forcing notion ordered by reverse inclusion, i.e., $p \leq q$ iff $p \supseteq q$ (We denote by $p: A\rightharpoonup B$ a partial function from $A$ to $B$).

\textbf{Group of permutations $\mathcal{G}_{1}$ of $\mathbb{P}_{1}$:} Let $\mathcal{G}_{1}$ be the full permutation group of $\omega$. Extend $\mathcal{G}_{1}$ to an automorphism group of $\mathbb{P}_{1}$ by letting an $a \in \mathcal{G}_{1}$ act on a $p \in \mathbb{P}_{1}$ by $a^{*}(p)= \{(n, a(i), \beta) : (n, i, \beta) \in p\}$. We identify $a^{*}$ with $a \in \mathcal{G}_{1}$. We can see that this is an automorphism group of $\mathbb{P}_{1}$.

\textbf{Normal filter $\mathcal{F}_{1}$ of subgroups over $\mathcal{G}_{1}$}: For every $n \in \omega$ define the following sets:
\begin{equation}
E_{n} =\{p \cap (n \times \omega \times \omega_{n}): p \in \mathbb{P}_{1}\},
\text{fix}_{\mathcal{G}_{1}}(E_{n}) = \{a \in \mathcal{G}_{1}: \forall p \in E_{n}(a(p) = p)\}.
\end{equation}
We can see that $\mathcal{F}_{1} =\{X \subseteq \mathcal{G}_{1} : \exists n \in \omega,$ fix$_{\mathcal{G}_{1}}(E_{n}) \subseteq X\}$ is a normal filter of subgroups over $\mathcal{G}_{1}$.

Feferman--L\'{e}vy's symmetric extension is the symmetric extension obtained by $\langle\mathbb{P}_{1}, \mathcal{G}_{1},\mathcal{F}_{1}\rangle$ where the statement `$\aleph_{1}$ is regular' \cite[Form 34]{HR1998} fails. It is known that the following statements follow from `$\aleph_{1}$ is regular' in $\mathsf{ZF}$ (cf. \cite{Ban2,BG1}).

(*): If $P$ is a poset such that the underlying set has a well-ordering and if all antichains in $P$ are finite and all chains in $P$ are countable, then $P$ is countable.

(**): If $P$ is a poset such that the underlying set has a well-ordering and if all antichains in $P$ are countable and all chains in $P$ are finite, then $P$ is countable.

\begin{question}
Does any of (**) and (*) is true in Feferman--L\'{e}vy's symmetric extension?
\end{question}
%%%%%%%%%%%%%%%
\section{Acknowledgements}
I am very thankful to the anonymous referee for careful reading of the paper and for useful suggestions which
improved its quality.
%%%%%%%%%%%%%%%

\end{document}